\documentclass[11pt]{article}
\usepackage{epigamath}

\usepackage[notext]{kpfonts}
\usepackage{baskervald}

\setpapertype{A4}

\usepackage[english]{babel}

\usepackage[dvips]{graphicx}     

\removelength{1cm}

\title{On the B-Semiampleness Conjecture}
\titlemark{On the B-Semiampleness Conjecture}
\author{Enrica Floris and Vladimir Lazi\'c}
\authoraddresses{
\authordata{Enrica Floris}{\firstname{Enrica} \lastname{Floris}\\
\institution{Universit\'e de Poitiers, Laboratoire de Math\'ematiques et Applications, UMR~CNRS 7348, T\'el\'eport 2, Boulevard Marie et Pierre Curie, BP 30179, 86962 Futuroscope Chasseneuil Cedex, France}\\
\email{enrica.floris@univ-poitiers.fr}} \\
\authordata{Vladimir Lazi\'c}{\firstname{Vladimir} \lastname{Lazi\'c}\\
\institution{Fachrichtung Mathematik, Campus, Geb\"aude E2.4, Universit\"at des Saarlandes, 66123 Saarbr\"ucken, Germany}\\
\email{lazic@math.uni-sb.de}}
}
\authormark{E. Floris and V. Lazi\'c}
\date{\vspace{-5ex}} 
\journal{\'Epijournal de G\'eom\'etrie Alg\'ebrique} 
\acceptation{Received by the Editors on January 7, 2019, and in final form on June 4, 2019. \\ Accepted on July 21, 2019.}

\newcommand{\articlereference}{Volume 3 (2019), Article Nr. 12}

\acknowledgement{Lazi\'c was supported by the DFG-Emmy-Noether-Nachwuchsgruppe ``Gute Strukturen in der h\"oherdimensionalen birationalen Geometrie". During the initial stage of this project Floris was funded by the Max Planck Institute for Mathematics in Bonn. We would like to thank Y.\ Brunebarbe, P.\ Cascini, Y.\ Gongyo, A.\ H\"oring and N.\ Tsakanikas for useful conversations related to this work and the referee for many comments which improved the
quality of the paper considerably.}

\usepackage[all]{xy}

\allowdisplaybreaks

\newdimen\origiwspc
\origiwspc=\fontdimen2\font

\usepackage{float}

\newtheorem{thm}{Theorem}[section]

\newtheorem{pro}[thm]{Proposition}

\newtheorem{lem}[thm]{Lemma}

\newtheorem{con}[thm]{Conjecture}

\newtheorem{thmA}{Theorem}

\newtheorem{corA}[thmA]{Corollary}

\newtheorem{dfn}[thm]{Definition}

\newtheorem{rem}[thm]{Remark}

\newcommand{\N}{\mathbb{N}}
\newcommand{\C}{\mathbb{C}}
\newcommand{\R}{\mathbb{R}}
\newcommand{\Q}{\mathbb{Q}}

\newcommand{\OO}{\mathcal{O}}

\newcommand{\red}{\mathrm{red}}

\DeclareMathOperator{\rk}{rk}

\DeclareMathOperator{\codim}{codim}
\DeclareMathOperator{\Bs}{Bs}

\DeclareMathOperator{\Exc}{Exc}
\DeclareMathOperator{\mult}{mult}

\DeclareMathOperator{\Supp}{Supp}

\DeclareMathOperator{\ddiv}{div}
\DeclareMathOperator{\sB}{\mathbf{B}}

\begin{document}


\maketitle



\begin{prelims}

\vspace{-0.55cm}

\def\abstractname{Abstract}
\abstract{The B-Semiampleness Conjecture of Prokhorov and Sho\-ku\-rov predicts that the moduli part in a canonical bundle formula is semiample on a birational modification. We prove that the restriction of the moduli part to any sufficiently high divisorial valuation is semiample, assuming the conjecture in lower dimensions.}

\keywords{Canonical bundle formula, moduli divisor}

\MSCclass{14H10; 14E30; 14N30}


\languagesection{Fran\c{c}ais}{%

\vspace{-0.05cm}
{\bf Titre. Sur la conjecture de B-semi-amplitude} \commentskip {\bf R\'esum\'e.} La conjecture de B-semi-amplitude de Prokhorov et Shokurov pr\'edit que la partie modulaire de la formule du fibr\'e canonique doit \^etre semi-ample sur une modification birationnelle. En supposant la validit\'e de cette conjecture en dimensions inf\'erieures, nous montrons que la partie modulaire est semi-ample en restriction \`a toute valuation divisorielle suffisamment haute.}

\end{prelims}


\newpage

\setcounter{tocdepth}{1} \tableofcontents

\section{Introduction}

Let $(X, \Delta)$ be a complex projective pair with log canonical singularities and let $f\colon X\to Y$ be a morphism such that
\begin{equation}\label{eq:1}
K_X+\Delta\sim_\Q f^*D
\end{equation}
for some $\Q$-Cartier $\Q$-divisor $D$ on $Y$. We say that $f$ is an \emph{lc-trivial fibration}; see Section \ref{sec:cbf} below. A typical example is when $K_X+\Delta$ is semiample and $f$ is the associated Iitaka fibration; a plethora of similar situations occurs in algebraic geometry. It is a fundamental question whether there exists a log canonical structure $(Y,\Delta_Y)$ such that $D\sim_\Q K_Y+\Delta_Y$: in other words, whether the singularities of $X$ \emph{descend} to $Y$.

This is an important question for at least two reasons: first, an affirmative answer would show that log canonical singularities form a ``stable'' category, and second, it enables proofs by induction.

The affirmative answer to the question above is known in several important special cases: when $f$ is a fibration and $(X,\Delta)$ has klt singularities \cite{Amb05a}, when $f$ is generically finite \cite{FG12}, and when $Y$ is a curve  \cite[Theorem 0.1]{Amb04}.

With notation as in \eqref{eq:1}, it is known that
$$D\sim_\Q K_Y+B_Y+M_Y,$$
where $B_Y$ (the \emph{discriminant}) is closely related to the singularities of $f$, and the divisor $M_Y$ (the \emph{moduli divisor}) conjecturally carries information on the birational variation of the fibres of $f$. A study of formulas of this type -- of \emph{canonical bundle formulas} -- began with Kodaira's canonical bundle formula for elliptic surface fibrations.

Much is known about the birational behaviour of such formulas: In particular, it is known that, after passing to a certain birational model $Y'$ of $Y$, the divisor $M_{Y'}$ is nef and for any higher birational model $Y''\to Y'$ the induced $M_{Y''}$ on $Y''$ is the pullback of $M_{Y'}$ \cite{Kaw98,Amb04,Kol07}. We call such a variety $Y'$ an \emph{Ambro model} of $f$.

The following is a conjecture of Prokhorov and Shokurov \cite[Conjecture 7.13]{PS09}, and it is the most important open problem regarding canonical bundle formulas.

\paragraph{B-Semiampleness Conjecture.}
{\em Let $(X,\Delta)$ be a pair and let $f\colon (X,\Delta)\to Y$ be an lc-trivial fibration to an $n$-dimensional variety $Y$, where the divisor $\Delta$ is effective over the generic point of $Y$. If $Y$ is an Ambro model of $f$, then the moduli divisor $M_Y$ is semiample.}

\bigskip

A proof of this conjecture would give an affirmative answer to the question stated at the beginning of this paper. Note that when the singularities of $(X,\Delta)$ are only klt, it was sufficient to show a weaker version -- that the moduli part is \emph{nef and b-good}, as demonstrated by Ambro in \cite{Amb05a}.

The B-Semiampleness Conjecture is known when $Y$ is a curve \cite[Theorem 0.1]{Amb04}, when a general fibre of $f$ is a curve by the classical work of Kodaira and by \cite[Theorem 8.1]{PS09}, or when a general fibre of $f$ is a smooth non-rational surface \cite{Fuj03,Fil18}; see also \cite{BC16} showing the semiampleness of a ``nearby'' divisor. Often a crucial role in the proof is played by the existence of a moduli space for the fibres. Nothing has been previously known in general when $\dim Y\geq2$.

In the remainder of the paper, we say that the B-Semiampleness Conjecture holds in dimension $n$, if it holds (in the notation from formulation of the B-Semiampleness Conjecture) for all lc-trivial fibrations $f\colon(X,\Delta)\to Y$ with $\dim Y=n$.

\medskip

\noindent{\bf\large The content of the paper.}
The main result of this paper is that the moduli part of an lc-trivial fibration is semiample when restricted to any divisorial valuation over its Ambro model:

\begin{thmA}\label{thm:main}
Assume the B-Semiampleness Conjecture in dimension $n-1$.

Let $(X,\Delta)$ be a pair and let $f\colon (X,\Delta)\to Y$ be an lc-trivial fibration to an $n$-dimensional variety $Y$, where the divisor $\Delta$ is effective over the generic point of $Y$. Assume that $Y$ is an Ambro model for $f$.

Then for every birational model $\pi\colon Y'\to Y$ and for every prime divisor $T$ on $Y'$ with the normalisation $T^\nu$ and the induced morphism $\nu\colon T^\nu\to Y'$, the divisor $\nu^*\pi^*M_Y$ is semiample on $T^\nu$.
\end{thmA}

If the moduli part $M_Y$ of the fibration is big and we consider only components of the locus $\sB_+(M_Y)$ where it is not ample (the \emph{augmented base locus} of $M_Y$, see Definition \ref{dfn:baseloci}), we can relax the assumptions on the B-Semiampleness Conjecture:

\begin{thmA}\label{thm:bplus}
Assume the B-Semiampleness Conjecture in dimensions at most $n-2$.

Let $(X,\Delta)$ be a pair and let $f\colon (X,\Delta)\to Y$ be an lc-trivial fibration to an $n$-dimensional variety $Y$, where the divisor $\Delta$ is effective over the generic point of $Y$. Assume that $Y$ is an Ambro model for $f$ and that $M_Y$ is big.

Then for every birational model $\pi\colon Y'\to Y$ and for any divisorial component $T$ of $\sB_+(\pi^*M_Y)$ with the normalisation $T^\nu$ and the induced morphism $\nu\colon T^\nu\to Y'$, the divisor $\nu^*\pi^*M_Y$ is semiample on $T^\nu$.
\end{thmA}

Immediate corollaries are:

\begin{corA}\label{cor:surfaces}
Let $(X,\Delta)$ be a pair and let $f\colon (X,\Delta)\to Y$ be an lc-trivial fibration to a surface $Y$, where the divisor $\Delta$ is effective over the generic point of $Y$. Assume that $Y$ is an Ambro model for $f$.

Then for every divisor $T$ on $Y$ with the normalisation $T^\nu$ and the induced morphism $\nu\colon T^\nu\to Y$, the divisor $\nu^*M_Y$ is semiample.

If additionally $M_Y$ is big, then $\nu^*M_Y\sim_\Q0$ for every divisorial component $T$ of $\sB_+(M_Y)$ with the normalisation $T^\nu$ and the induced morphism $\nu\colon T^\nu\to Y$.
\end{corA}

\begin{corA}\label{cor:threefolds}
Let $(X,\Delta)$ be a pair and let $f\colon (X,\Delta)\to Y$ be an lc-trivial fibration to a threefold $Y$, where the divisor $\Delta$ is effective over the generic point of $Y$. Assume that $Y$ is an Ambro model for $f$ and that $M_Y$ is big.

Then $\nu^*M_Y$ is semiample for every divisorial component $T$ of $\sB_+(M_Y)$ with the normalisation $T^\nu$ and the induced morphism $\nu\colon T^\nu\to Y$.
\end{corA}

\medskip

\noindent{\bf\large Sketch of the proof.} 
The proof of Theorem \ref{thm:main} is very technical, and the core of the arguments is contained in Section \ref{sec:5and6}. In Proposition \ref{pro:step1aa} we use the Minimal Model Program to achieve, roughly, the situation where over $T$ there exists a log canonical centre $S$ of $(X,\Delta)$ and an lc-trivial fibration $h\colon (S,\Delta_S)\to T$ whose moduli part is \emph{almost} $M_Y|_T$. This part of the proof is somewhat similar to and inspired by the proof in \cite{FG14}. An important difference is that we do not cut down to curves in order to use the semistable reduction, but we study in detail those divisors which are not contracted by the MMP.

To this end, we introduce \emph{acceptable} lc-trivial fibrations; roughly speaking, these fibrations might not come with a sub-boundary $\Delta$ which is effective on a general fibre of the lc-trivial fibration, but are obtained from such a fibration by blowing up. This suffices to ensure that, by the results of Nakayama \cite{Nak04}, the MMP that we run actually terminates.

Finally, a careful choice of a base change in Proposition \ref{pro:step1ab} and an analysis of the ramification loci of the finite part of the Stein factorisation of the map $f|_S$ allow to finish the proof.

We mention here that this all uses many foundational results of Kawamata and Ambro.

\medskip

\noindent{\bf\large A reduction result.}
As a by-product of the techniques employed in the proofs, in Section \ref{sec:reduction} we reduce the B-Semiampleness Conjecture to a much weaker Conjecture \ref{con:weakbsemi}.

\begin{thmA}\label{thm:reductiontoweaker}
Assume that the B-Semiampleness Conjecture holds for klt-trivial fibrations $f\colon (X,\Delta)\to Y$, where $(X,\Delta)$ is a log canonical pair, $\Delta\geq0$, the moduli divisor $M_Y$ is big and $\dim Y\leq n$.

Then the B-Semiampleness Conjecture holds in dimension $n$.
\end{thmA}

The result says that it suffices to prove the B-Semiampleness Conjecture for \emph{klt-trivial} fibrations $f\colon (X,\Delta)\to Y$, where $\Delta$ is effective and the moduli divisor of $f$ is big. The method of the proof is similar to that in Section \ref{sec:5and6}, together with a quick application of a result of Ambro \cite[Theorem 3.3]{Amb05a}. This improves on the discussion and results in \cite[Section 3]{Fuj15}.

Consequently, in Theorems \ref{thm:main} and \ref{thm:bplus} it suffices to assume Conjecture \ref{con:weakbsemi} instead of the B-Semiampleness Conjecture.

\section{Preliminaries}

We work over $\C$. We denote by $\equiv$, $\sim$ and $\sim_\Q$ the numerical, linear and $\Q$-linear equivalence of divisors respectively. For a decomposition of a Weil $\Q$-divisor $D=\sum d_i D_i$ into prime components and a real number $\alpha$, we denote
$$D_\red=\sum D_i,\quad D^{\geq\alpha}:=\sum_{d_i\geq\alpha}d_i D_i\quad\text{and}\quad D^{\leq\alpha}:=\sum_{d_i\leq\alpha}d_i D_i,$$
and similarly for $D^{>\alpha}$, $D^{<\alpha}$ and $D^{=\alpha}$. If $f\colon X\to Y$ is a proper surjective morphism between normal varieties and $D$ is a Weil $\R$-divisor on $X$, then $D_v$ and $D_h$ denote the vertical and the horizontal part of $D$ with respect to $f$; the relevant map will be clear from the context. For instance, the notation $D_v^{\geq0}$ denotes the non-negative part of the vertical part of $D$.

\subsection{Base loci}

We start with the following well-known result; we include the proof for the benefit of the reader.

\begin{lem}\label{geq2}
Let $f\colon X\rightarrow Y$ be a surjective morphism between normal projective varieties. Let $D$ be a Cartier divisor on $Y$. Then $D$ is semiample if and only if $f^*D$ is semiample.
\end{lem}

\begin{proof}
Since necessity is clear, it suffices to prove sufficiency. So assume that $f^*D$ is semiample, and we may assume that $f^*D$ is basepoint free. By considering the Stein factorisation of $f$, it suffices to consider separately the cases when either $f$ has connected fibres or $f$ is finite.

When $f$ has connected fibres, pick a closed point $y\in Y$ and a closed point $x\in f^{-1}(y)$. Since $f^*D$ is basepoint free, there exists $D_X\in|f^*D|$ such that $x\notin\Supp D_X$. As $H^0(X,f^*D)\simeq H^0(Y,D)$, there exists a divisor $D_Y\in|D|$ such that $D_X=f^*D_Y$. Then it is clear that $y\notin\Supp D_Y$, which shows that $D$ is basepoint free.

Now assume that $f$ is finite. If $X^\circ$ and $Y^\circ$ are the smooth loci of $X$ and $Y$, respectively, consider the open sets $U_Y:=Y^\circ\setminus f(X\setminus X^\circ)\subseteq Y$ and $U_X:=f^{-1}(U_Y)\subseteq X$. By normality of $X$ and $Y$, we have
$$\codim_X(X\setminus U_X)=\codim_Y(Y\setminus U_Y)\geq2,$$
and the map $f|_{U_X}\colon U_X\to U_Y$ is flat by \cite[Exercise III.9.3(a)]{Har77}. Then, by \cite[Example 1.7.4]{Ful98}, it yields $(f|_{U_X})_*(f|_{U_X})^*(D|_{U_Y})=(\deg f)(D|_{U_Y})$  hence
\begin{equation}\label{eq:pushpull}
f_*f^*D=(\deg f)D.
\end{equation}
Now, for any point $y\in Y$, a general element $E\in|f^*D|$ avoids the finite set $f^{-1}(y)$, hence $y\notin\Supp f_*E$. Since $f_*E\in |(\deg f)D|$ by \eqref{eq:pushpull}, this shows that $(\deg f)D$ is basepoint free.
\end{proof}

We recall the definition of stable and augmented base loci of $\Q$-divisors.

\begin{dfn}\label{dfn:baseloci}
{\rm Let $D$ be a $\Q$-Cartier $\Q$-divisor on a normal projective variety $X$. If $D$ is integral, the base locus of $D$ is denoted by $\Bs|D|$. The \emph{stable base locus of $D$} is
$$\sB(D)=\bigcap_{D\sim_\Q D'\geq0}\Supp D',$$
and the \emph{augmented base locus of $D$} is
$$\sB_+(D)=\bigcap_{n\in\N_{>0}} \textstyle\sB\big(D-\frac1nA\big),$$
where $A$ is an ample divisor on $X$; this definition does not depend on the choice of $A$. Clearly $\sB(D)\subseteq\sB_+(D)$, and both $\sB(D)$ and $\sB_+(D)$ are closed subsets of $X$. The set $\sB_+(D)$ is empty if and only if $D$ is ample, and is different from $X$ if and only if $D$ is big.}
\end{dfn}

We need the following lemma in the proof of Theorem \ref{thm:bplus}.

\begin{lem}\label{lem:nakamaye}
Let $X$ be a projective manifold and let $D$ be a nef and big $\Q$-divisor on $X$. Let $T\subseteq\sB_+(D)$ be a prime divisor. Then $\OO_T(D)$ is not big.
\end{lem}

\begin{proof}
Arguing by contradiction, we assume that $\mathcal L:=\OO_T(D)$ is big. Fix a point $x\in T\setminus\sB_+(\mathcal L)$ which does not belong to any other component of $\sB_+(D)$. Since $x\in\sB_+(D)$, by \cite[Theorem 0.3]{Nak00} there exists a positive dimensional subvariety $V\subseteq\sB_+(D)\subseteq X$ such that $x\in V$ and $D^{\dim V}\cdot V=0$. By the choice of $x$ we necessarily have $V\subseteq T$, and hence
$$\mathcal L^{\dim V}\cdot V=0.$$
But then $V\subseteq\sB_+(\mathcal L)$ again by \cite[Theorem 0.3]{Nak00}, hence $x\in \sB_+(\mathcal L)$, a contradiction.
\end{proof}

\subsection{Pairs and resolutions}\label{subsec:pairs}

A \emph{pair} $(X,\Delta)$ consists of a normal variety $X$ and a Weil $\Q$-divisor $\Delta$ such that $K_X+\Delta$ is $\Q$-Cartier. A pair $(X,\Delta)$ is \emph{log smooth} if $X$ is smooth and the support of $\Delta$ is a simple normal crossings divisor. In this paper, unless explicitly stated otherwise, we do not require that $\Delta$ is an effective divisor.

A \emph{log resolution} of a pair $(X,\Delta)$ is a birational morphism $f\colon Y\to X$ such that $Y$ is smooth, the exceptional locus $\Exc(f)$ is a divisor and the divisor $f_*^{-1}\Delta+\Exc(f)$ has simple normal crossings support.

In this paper we use the term \emph{embedded resolution} of a pair $(X,\Delta)$ in the following strong sense: it is a log resolution of the pair which is an isomorphism on the locus where $X$ is smooth and $\Delta$ has simple normal crossings. The existence of embedded resolutions in this sense was proved in \cite{Sza94}, see also \cite[Theorem 10.45(2)]{Kol13}.

We will need the following lemma in the proof of Proposition \ref{pro:step1ab}.

\begin{lem}\label{lem:blowup}
Let $\pi\colon X'\to X$ be a finite morphism between normal proper varieties, and let $\widetilde f\colon \widetilde Y\to X'$ be a birational morphism. Then there exists a birational morphism $f\colon Y\to X$ such that, if $Y'$ is the normalisation of the main component of the fibre product $X'\times_X Y$ and $f'\colon Y'\dashrightarrow \widetilde Y$ is the induced birational map, then $(f')^{-1}$ is an isomorphism at the generic point of each $\widetilde f$-exceptional prime divisor on $\widetilde Y$.
$$
\xymatrix{ X' \ar[d]_\pi & \widetilde Y \ar[l]_{\quad \widetilde f} & Y' \ar@{-->}[l]_{\quad f'} \ar[d] \\
            X && Y \ar[ll]_f
 }
$$
Moreover, we may assume that $f$ is a composition of blowups along proper subvarieties.
\end{lem}

\begin{proof}
Each exceptional divisor on $\widetilde Y$ gives a divisorial valuation $E_i'$ over $X'$. By the proof of \cite[Proposition 2.14(i)]{DL15} there are geometric valuations $E_i$ over $X$ such that the following holds:

Let $f\colon Y\to X$ be any birational model such that all valuations $E_i$ are realised as prime divisors on $Y$; we can assume that $f$ is a sequence of blowups along proper subvarieties by \cite[Lemma 2.45]{KM98}. If $Y'$ is the normalisation of the main component of the fibre product $X'\times_X Y$, then every valuation $E_i'$ is a prime divisor on $Y'$.

Now, since the $E_i'$ are prime divisors on both $\widetilde{Y}$ and $Y'$, the induced birational map $f'\colon Y'\dashrightarrow \widetilde Y$ has to be an isomorphism at the corresponding generic points, which proves the lemma.
\end{proof}

\begin{dfn}
{\rm Let $(X,\Delta)$ be a pair and let $\pi\colon Y\to X$ be a birational morphism, where $Y$ is normal. We can write
$$K_Y\sim_\Q \pi^*(K_X+\Delta)+\sum a(E_i,X,\Delta)\cdot E_i,$$
where $E_i\subseteq Y$ are distinct prime divisors and the numbers $a(E_i,X,\Delta)\in\Q$ are \emph{discrepancies}. The closure of the image of a geometric valuation $E$ on $X$ is the \emph{centre of $E$ on $X$}, denoted by $c_X(E)$. The pair $(X,\Delta)$ is \emph{klt}, respectively \emph{log canonical}, if $a(E,X,\Delta)>-1$, respectively $a(E,X,\Delta) \geq -1$, for every geometric valuation $E$ over $X$.

The pair $(X,\Delta)$ is \emph{dlt} if $\Delta\geq0$ and there is a closed subset $Z\subseteq X$ such that $X\setminus Z$ is smooth, $\Delta\vert_{X\setminus Z}$ is a simple normal crossings divisor, and if $c_X( E) \subseteq Z$ for a geometric valuation $E$ over $X$, then $a(E,X,\Delta)>-1$.

A closed subset $Z\subseteq X$ is a \emph{log canonical centre} of a log canonical pair $(X,\Delta)$ if $Z=c_X( E)$ for some geometric valuation $E$ over $X$ with $a(E,X,\Delta)=-1$.

Let $f\colon (X,\Delta)\to Y$ be a proper morphism and let $Z\subseteq Y$ be a closed subset. Then $(X,\Delta)$ is klt, respectively log canonical, over $Z$ (or over the generic point of $Z$) if $a(E,X,\Delta)>-1$, resp.\ $a(E,X,\Delta) \geq -1$, for every geometric valuation $E$ over $X$ with $f\big(c_X(E)\big)=Z$. We say that a closed subvariety $S$ of $X$ is a \emph{minimal log canonical centre of $(X,\Delta)$ over $Z$} if $S$ is a minimal log canonical centre of $(X,\Delta)$ (with respect to inclusion) which dominates $Z$.}
\end{dfn}

The following is \cite[Proposition 3.9.2]{Fuj07c}, see also the proof of \cite[Theorem 4.16]{Kol13}.

\begin{pro}\label{pro:fujino}
Let $(X,\Delta)$ be a dlt pair. Let $\Delta^{=1}= \sum_{i\in I} D_i$ be the decomposition into irreducible components. Then $S$ is a log canonical centre of $(X,\Delta)$ with $\codim_X S = k$ if and only if $S$ is an irreducible component of $D_{i_1} \cap D_{i_2}\cap\dots\cap D_{i_k}$ for some $\{i_1, i_2,\ldots, i_k\} \subseteq I$. Moreover, $S$ is normal and it is a log canonical centre of the dlt pair $\big(D_{i_1},(\Delta-D_{i_1})|_{D_{i_1}}\big)$.
\end{pro}

\begin{dfn}
{\rm Let $(X_1, \Delta_1)$ and $(X_2, \Delta_2)$ be two pairs. A birational map $\theta\colon X_1\dasharrow X_2$ is \emph{crepant birational} if $a(E,X_1,\Delta_1)=a(E,X_2,\Delta_2)$ for every geometric valuation $E$ over $X_1$ and $X_2$. If additionally we have a commutative diagram
$$
\xymatrix{
X_1\ar[rd]_{f_1} \ar@{-->}[rr]^{\theta}&&X_2\ar[ld]^{f_2}\\
&Y&
}
$$
for some variety $Y$, then we say that $\theta$ is \emph{crepant birational over $Y$}.}
\end{dfn}

We will need the following lemma in the proof of Proposition \ref{pro:step1aa}.

\begin{lem}\label{lem:1to1}
Let $(X,\Delta)$ be a dlt pair such that $K_X+\Delta\sim_\Q F$, where $F$ is an effective $\Q$-divisor having no common components with $\Delta^{=1}$. Assume that there exists a smooth open subset $U\subseteq X$ which intersects every log canonical centre of $(X,\Delta)$, and such that the divisor $(\Delta+F)|_U$ has simple normal crossings support.

Then any $(K_X+\Delta)$-MMP is an isomorphism at the generic point of each log canonical centre of $(X,\Delta)$, and moreover, it induces an inclusion-preserving bijection from the set of log canonical centres of $(X,\Delta)$ to the set of log canonical centres at each step of the MMP.
\end{lem}

\begin{proof}
For $i\geq0$ let $(X_i,\Delta_i)$ be the pairs in the steps of this MMP with $(X_0,\Delta_0):=(X,\Delta)$, and let $F_i$ be the strict transform of $F$ on $X_i$.

\medskip

\emph{Step 1.}
Let $\theta_i\colon X_i\to Z_i$ be the extremal contraction at the $i$-th step of the MMP. If $\Gamma_i$ is a curve on $X_i$ contracted by $\theta_i$, then
$$F_i\cdot\Gamma_i=(K_{X_i}+\Delta_i)\cdot\Gamma_i<0,$$
and hence $\Exc(\theta_i)\subseteq \Supp F_i$.

\medskip

\emph{Step 2.}
We now show by induction on $i$ that:
\begin{enumerate}
\item[(a)$_i$] $\theta_i$ is an isomorphism at the generic point of each log canonical centre of $(X_i,\Delta_i)$, and
\item[(b)$_i$] there exists a smooth open subset $U_i\subseteq X_i$ containing the generic point of each log canonical centre of $(X_i,\Delta_i)$ such that $(\Delta_i+F_i)|_{U_i}$ has simple normal crossings support.
\end{enumerate}
Indeed, note that (b)$_0$ holds by assumption. Next, assume that (b)$_i$ holds for some index $i$. Let $P$ be a log canonical centre of $(X_i,\Delta_i)$. Then the generic point of $P$ belongs to $U_i$ and $P$ is an irreducible component of an intersection of components of $\Delta_i^{=1}$ by Proposition \ref{pro:fujino}, hence $P$ cannot be contained in $\Supp F_i$ by (b)$_i$. By Step 1, the map $\theta_i$ is an isomorphism at the generic point of $P$, which shows (a)$_i$. Therefore, by possibly shrinking $U_i$, we may assume that $\theta_i|_{U_i}$ is an isomorphism, and we define $U_{i+1}$ as $\theta_i(U_i)$ if $\theta_i$ is divisorial, or as the image of $U_i$ by the corresponding flip if $\theta_i$ is a flipping contraction. This shows (b)$_{i+1}$.

\medskip

\emph{Step 3.}
Finally, the last statement in the lemma follows immediately from Step 2, from the definition of dlt singularities, and the fact that the MMP does not decrease discrepancies.
\end{proof}

\subsection{Weakly exceptional divisors}

We use the relative Nakayama-Zariski decomposition of pseudoeffective divisors as in \cite[Chapter III]{Nak04}. Note that by \cite{Les16} this is not always well-defined; however, in all the cases we consider in this paper, the decomposition exists and behaves as in the absolute case. Note that one can define this decomposition on any $\Q$-factorial variety; below we give additional comments when we use it on non-smooth varieties.

The following definition is crucial for applications in Sections \ref{sec:5and6} and \ref{sec:reduction}. Note that part (b) differs somewhat from that in \cite{Nak04}, see Remark \ref{rem:weakly}.

\begin{dfn}\label{insuff}
{\rm Let $f\colon X\rightarrow Y$ be a projective surjective morphism of normal varieties and let $D$ be an effective Weil $\R$-divisor on $X$ such that $f(D)\neq Y$. Then $D$ is:
\begin{enumerate}
\item[(a)] \emph{$f$-exceptional} if $\codim_Y\Supp f(D)\geq2$,
\item[(b)] \emph{of insufficient fibre type over $Y$} if $f(D)$ has pure codimension $1$ and for every prime divisor $\Gamma\subseteq f(D)$ there exists a divisor $E\not\subseteq\Supp D$ such that $f(E)=\Gamma$.
\item[(c)] \emph{weakly $f$-exceptional} if there are effective divisors $D_1$ and $D_2$ such that $D=D_1+D_2$, where $D_1$ is $f$-exceptional and $D_2$ is of insufficient fibre type over $Y$.
\end{enumerate}
Let additionally $g\colon Y\rightarrow Z$ be a projective surjective morphism of normal varieties and let $D_h$ and $D_v$ denote the horizontal and vertical parts of $D$ with respect to $g\circ f$. If $D_h$ is $f$-exceptional and $D_v$ is weakly $(g\circ f)$-exceptional, then we call $D$ an \emph{$(f,g)$-EWE divisor}.}
\end{dfn}

\begin{rem}
{\rm The abbreviation EWE stands for ``exceptional-weakly-exceptional''.}
\end{rem}

Proposition \ref{pro:EWE} shows the most important property of EWE divisors.

\begin{lem}\label{VE55}
Let $f\colon X\to Y$ and  $g\colon Y\rightarrow Z$ be projective surjective morphisms with connected fibres between normal varieties, assume that $X$ is smooth, and let $D$ be an effective $(f,g)$-EWE divisor on $X$. Then there exists a component $\Gamma$ of $D$ such that $\OO_\Gamma(D)$ is not $(g\circ f)|_\Gamma$-pseudoeffective over $g\big(f(\Gamma)\big)$.
\end{lem}

\begin{proof}
Let $D_h$ and $D_v$ denote the horizontal and vertical parts of $D$ with respect to $g\circ f$. Assume first that $D_h=0$. If $D_v$ is $(g\circ f)$-exceptional, we are done by \cite[Lemma III.5.1]{Nak04}. Otherwise, choose a prime divisor $\Gamma\subseteq\Supp D_v$ such that $g\big(f(\Gamma)\big)$ is a divisor. Then $\OO_\Gamma(D)$ is not $(g\circ f)|_\Gamma$-pseudoeffective over $g\big(f(\Gamma)\big)$ by \cite[Lemma III.5.2]{Nak04}.

Thus we may assume that $D_h\neq0$. By \cite[Lemma III.5.1]{Nak04} there is a component $\Gamma$ of $D_h$ such that $\OO_\Gamma(D_h)$ is not $(f|_\Gamma)$-pseudoeffective over $f(\Gamma)$, hence $\OO_\Gamma(D_h)$ is not $(g\circ f)|_\Gamma$-pseudoeffective over $g\big(f(\Gamma)\big)=Z$. Since $D_v$ does not intersect the generic fibre of $g\circ f$, we obtain that $\OO_\Gamma(D_h+D_v)$ is not $(g\circ f)|_\Gamma$-pseudoeffective over $Z$, which implies the lemma.
\end{proof}

\begin{pro}\label{pro:EWE}
Let $f\colon X\to Y$ and  $g\colon Y\rightarrow Z$ be projective surjective morphisms of normal varieties, assume that $X$ is smooth, and let $D$ be an effective $(f,g)$-EWE divisor on $X$. Then $D=N_{\sigma}(D;X/Z)$.
\end{pro}

\begin{proof}
The proof goes verbatim as the proof of \cite[Proposition III.5.7]{Nak04}, replacing \cite[Lemma III.5.5]{Nak04} with Lemma \ref{VE55}.
\end{proof}

\begin{rem}\label{rem:weakly}
{\rm The definition of insufficient fibre type above is slightly different from (and more precise than) the one in \cite[\S III.5.a]{Nak04}. In order for \cite[Corollary III.5.6 and Proposition III.5.7]{Nak04} to hold, one needs to work with the definition above; this is in fact implicit from the proofs of these two statements in op.\ cit. If one works with the definition as in op.\ cit., one can easily construct a counterexample: let $f\colon X\to Y$ be a fibration from a smooth surface to a smooth curve, let $F_1$ and $F_2$ be two distinct fibres of $f$ and let $\pi\colon \widetilde X\to X$ be the blowup of a point in $F_1$. Then $\pi_*^{-1}F_1+\pi^*F_2$ would be of insufficient fibre type over $Y$, and \cite[Proposition III.5.7 and Lemma III.4.2]{Nak04} would imply that $N_\sigma(\pi^*F_2;Z/Y)=\pi^*F_2$, a contradiction since $\pi^*F_2$ is $(f\circ\pi)$-nef.}
\end{rem}

We need in this paper the MMP with scaling as described in \cite[Definition 3.2]{Bir10}, and the fact that log canonical flips exist \cite{Bir12,HX13}. We also need the following.

\begin{lem}\label{lem:druel}
Let $(X,\Delta)$ be a $\Q$-factorial dlt pair and let $\pi\colon X\to U$ be a projective morphism such that $K_X+\Delta$ is $\pi$-pseudoeffective. If $N_\sigma(K_X+\Delta; X/U)$ is an $\R$-divisor, then any Minimal Model Program of $K_X+\Delta$ with scaling of an ample divisor over $U$ contracts precisely the components of $N_\sigma(K_X+\Delta; X/U)$.
\end{lem}

\begin{proof}
For $i\geq0$, let $(X_i,\Delta_i)$ be the pairs in a $(K_X+\Delta)$-MMP with scaling of an ample divisor over $U$, with $(X_0,\Delta_0):=(X,\Delta)$. Let $(p_i,q_i)\colon W_i\to X\times X_i$ be a smooth resolution of indeterminacies of the induced birational map $\varphi_i\colon X\dashrightarrow X_i$. Then there exists an effective $q_i$-exceptional divisor $E_i$ on $W_i$ such that
\begin{equation}\label{eq:100}
p_i^*(K_X+\Delta)\sim_\Q q_i^*(K_{X_i}+\Delta_i)+E_i.
\end{equation}
Analogously as in the proof of \cite[Lemma 2.16]{GL13} we have
\begin{equation}\label{eq:101}
N_\sigma\big(q_i^*(K_{X_i}+\Delta_i)+E_i;W_i/U\big)=N_\sigma\big(q_i^*(K_{X_i}+\Delta_i);W_i/U\big)+E_i,
\end{equation}
and the first two lines of the proof of \cite[Lemma III.5.15]{Nak04} give
\begin{equation}\label{eq:103}
N_\sigma(K_X+\Delta;X/U)=p_{i*}N_\sigma\big(p_i^*(K_X+\Delta);W_i/U\big)
\end{equation}
and
\begin{equation}\label{eq:102}
N_\sigma(K_{X_i}+\Delta_i;X_i/U)=q_{i*}N_\sigma\big(q_i^*(K_{X_i}+\Delta_i);W_i/U\big).
\end{equation}
Then:
\begin{align*}
N_\sigma(K_{X_i}+\Delta_i;X_i/U)&=q_{i*}N_\sigma\big(q_i^*(K_{X_i}+\Delta_i);W_i/U\big) & \text{by \eqref{eq:102}}\\
&=q_{i*}\big(N_\sigma\big(q_i^*(K_{X_i}+\Delta_i);W_i/U\big)+E_i\big) & \\
&=\varphi_{i*}p_{i*}N_\sigma\big(q_i^*(K_{X_i}+\Delta_i)+E_i;W_i/U\big) & \text{by \eqref{eq:101}}\\
&=\varphi_{i*}p_{i*}N_\sigma\big(p_i^*(K_X+\Delta);W_i/U\big) & \text{by \eqref{eq:100}}\\
&=\varphi_{i*}N_\sigma(K_X+\Delta;X/U). & \text{by \eqref{eq:103}}
\end{align*}
Since $K_{X_i}+\Delta_i$ is in the movable cone over $U$ if and only if $N_\sigma(K_{X_i}+\Delta_i;X_i/U)=0$, we conclude by \cite[Theorem 2.3]{Fuj11b}.
\end{proof}

\section{Canonical bundle formula}\label{sec:cbf}

In this section we define the main object of this paper and state several properties we repeatedly use.

\begin{dfn}
{\rm Let $(X,\Delta)$ be a pair and let $\pi\colon X'\rightarrow X$ be a log resolution of the pair. A morphism $f \colon (X,\Delta) \rightarrow Y$ to a normal projective variety $Y$ is a \emph{klt-trivial}, respectively \emph{lc-trivial}, fibration if $f$ is a surjective morphism with connected fibres, $(X,\Delta)$ has klt, respectively log canonical, singularities over the generic point of $Y$, there exists a $\Q$-Cartier $\Q$-divisor $D$ on $Y$ such that
$$K_X+\Delta\sim_\Q f^*D,$$
and if $f'=f\circ\pi$, then
$$\rk f'_*\OO_{X'}\big(\lceil K_{X'}-\pi^*(K_X+\Delta)\rceil\big) = 1,$$
respectively
$$\textstyle \rk f'_*\OO_{X'}\big(\lceil K_{X'}-\pi^*(K_X+\Delta)+\sum_{a(E,X,\Delta)=-1} E\rceil\big) = 1.$$}
\end{dfn}

\begin{rem}
{\rm This last condition in the previous definition is verified, for instance, if $\Delta$ is effective on the generic fibre, which is the case in this paper.}
\end{rem}

\begin{dfn}\label{dfn:cbf}
{\rm Let $f\colon (X,\Delta)\to Y$ be an lc-trivial fibration, and let $P\subseteq Y$ be a prime divisor. The \emph{generic log canonical threshold} of $f^*P$ with respect to $(X,\Delta)$ is
$$\gamma_P=\sup\{t\in\R\mid (X,\Delta+tf^*P) \textrm{ is log canonical over } P\}.$$
The \emph{discriminant} of $f$ is
\begin{equation}\label{discriminant}
B_f=\sum_P(1-\gamma_P)P.
\end{equation}
This is a Weil $\Q$-divisor on $Y$, and it is effective if $\Delta$ is effective. Fix $\varphi\in\C(X)$ and the smallest positive integer $r$ such that $K_X + \Delta +\frac{1}{r}\ddiv\varphi = f^*D$. Then there exists a unique Weil $\Q$-divisor $M_f$, the \emph{moduli part} of $f$, such that
\begin{equation}\label{cbf}
K_X + \Delta +\frac{1}{r}\ddiv\varphi = f^*(K_Y+B_f+M_f).
\end{equation}
The formula \eqref{cbf} is the \emph{canonical bundle formula} associated to $f$ (and $\varphi$).}
\end{dfn}

\begin{rem}
{\rm It is customary in the literature to denote the discriminant and the moduli part by $B_Y$ and $M_Y$. We adopted a different notation, since we sometimes have to compare discriminants or moduli parts of different lc-trivial fibrations which have the same base $Y$. However, if the fibration is clear from the context, we still occasionally write $B_Y$ and $M_Y$.}
\end{rem}

\begin{rem}\label{bbir}
{\rm If $f_1\colon (X_1,\Delta_1)\to Y$ and $f_2\colon (X_2,\Delta_2)\to Y$ are two lc-trivial fibrations over the same base  which are crepant birational over $Y$, then $B_{f_1}=B_{f_2}$ and $M_{f_1}\sim_\Q M_{f_2}$, see \cite[Lemma 2.6]{Amb04}.}
\end{rem}

\begin{rem}\label{rem:redtolc}
{\rm Let $f\colon(X,\Delta)\to Y$ be an lc-trivial (respectively klt-trivial) fibration, where $Y$ is smooth. Then there exists a divisor $\Delta^\circ$ on $X$ such that the pair $(X,\Delta^\circ)$ is log canonical and $f\colon(X,\Delta^\circ)\to Y$ is an lc-trivial (respectively klt-trivial) fibration. Indeed, let $\pi\colon X'\to X$ be a log resolution of $(X,\Delta)$, and set $\Delta':=K_{X'}-\pi^*(K_X+\Delta)$. Let $\{P_i\}_{i\in I}$ be the finite set of prime divisors on $X'$ such that $\mult_{P_i}\Delta'>1$. Then each $P_i$ is a vertical divisor, and let $D_i$ be any prime divisor on $Y$ which contains $f\big(\pi(P_i)\big)$. Then the divisor
$$\Delta^\circ:=\Delta-\sum_{i\in I}(\mult_{P_i}\Delta')f^*D_i$$
is the desired divisor.}
\end{rem}

The canonical bundle formula satisfies several desirable properties. To start with, if $f \colon(X,\Delta)\to Y$ is a klt-trivial (respectively lc-trivial) fibration, if $\rho\colon Y'\to Y$ is a proper generically finite morphism, and if we consider a base change diagram
$$
\xymatrix{
(X',\Delta') \ar[r]^\tau \ar[d]_{f'} & (X,\Delta)\ar[d]^{f}\\
Y' \ar[r]_{\rho}&Y,
}
$$
where $X'$ is the normalisation of the main component of $X\times_Y Y'$ and $\Delta':=K_{X'}-\tau^*(K_X+\Delta)$, then $f' \colon(X',\Delta')\to Y'$ is also a klt-trivial (respectively lc-trivial) fibration. In the rest of the paper, we implicitly refer to this klt-trivial (respectively lc-trivial) fibration when writing $B_{Y'}$ and $M_{Y'}$ for the discriminant and moduli part.

\medskip

The following is the important \emph{base change property} from \cite[Theorem 0.2]{Amb04}, \cite[Theorem 2]{Kaw98} and \cite[Theorem 3.6]{FG14}:

\begin{thm}\label{nefness}
Let $f \colon(X,\Delta)\to Y$ be an lc-trivial fibration. Then there exists a proper birational morphism $Y'\to Y$ such that for every proper birational morphism $\pi \colon Y''\to Y'$ we have:\begin{enumerate}
\item[(i)] $K_{Y'}+B_{Y'}$ is a $\Q$-Cartier divisor and $K_{Y''}+B_{Y''}=\pi^*(K_{Y'}+B_{Y'})$,
\item[(ii)] $M_{Y'}$ is a nef $\Q$-Cartier divisor and $M_{Y''}=\pi^*M_{Y'}$.
\end{enumerate}
\end{thm}

In the context of the previous theorem, we say that \emph{the moduli part descends} to $Y'$, and we call any such $Y'$ an \emph{Ambro model} for $f$.

One of the reasons why base change property is important is the following \emph{inversion of adjunction}.

\begin{thm}\label{thm:invAdjunction}
Let $f\colon(X,\Delta)\rightarrow Y$ be an lc-trivial fibration, and assume that $Y$ is an Ambro model for $f$. Then $(Y,B_Y)$ has klt, respectively log canonical, singularities in a neighbourhood of a point $y\in Y$ if and only if $(X,\Delta)$ has klt, respectively log canonical, singularities in a neighbourhood of $f^{-1}(y)$.
\end{thm}

\begin{proof}
This is  \cite[Theorem 3.1]{Amb04}. This result is stated for klt-trivial fibrations, but the proof extends verbatim to the lc-trivial case by using Theorem \ref{nefness}.
\end{proof}

Theorem \ref{thm:invAdjunction} says something highly non-trivial: that in certain situations, on an Ambro model the \emph{local} log canonical thresholds in Definition \ref{dfn:cbf} are actually \emph{global} log canonical thresholds.

The following is \cite[Theorem 3.3]{Amb05a}.

\begin{thm}\label{ambro1}
Let $f\colon(X,\Delta)\to Y$ be a klt-trivial fibration between normal projective varieties such that $\Delta$ is effective over the generic point of $Y$. Then there exists a diagram
$$
\xymatrix{
(X,\Delta)\ar[d]_f && (X',\Delta')\ar[d]^{f'}\\
Y & W \ar[l]^{\tau}\ar[r]_{\tau'}&Y'
}
$$
such that:
\begin{enumerate}
\item[(i)] $f'\colon(X',\Delta')\rightarrow Y'$ is a klt-trivial fibration,
\item[(ii)] $\tau$ is generically finite and surjective, and $\tau'$ is surjective,
\item[(iii)] the moduli part $M_{f'}$ is big, and after a birational base change we may assume that $f$ and $f'$ are Ambro models and $\tau^*M_f=\tau'^*M_{f'}$.
\end{enumerate}
\end{thm}

The following notions and the lemma will be used in the proof of Theorem \ref{KLTimplLC} and in our main technical results, Propositions \ref{pro:step1aa} and \ref{pro:step1ab}.

\begin{dfn}
{\rm Let  $f\colon(X, \Delta)\rightarrow Y$ be an lc-trivial fibration, where $(X,\Delta)$ is log smooth and $Y$ is smooth. Fix a prime divisor $T$ on $Y$. An \emph{$(f,T)$-bad divisor} is any reduced divisor $\Sigma_{f,T}$ on $Y$ which contains:
\begin{enumerate}
\item[(a)] the locus of critical values of $f$,
\item[(b)] the closed set $f(\Supp\Delta_v)\subseteq Y$, and
\item[(c)] the set $\Supp B_f\cup T$.
\end{enumerate}}
\end{dfn}

\begin{rem}
{\rm Clearly the set $\Sigma_{f,T}$ is not uniquely determined. It will be clear from the context which precise set we consider in the proofs below.}
\end{rem}

\begin{dfn}\label{dfn:acceptable}
{\rm Let $f\colon (X,\Delta)\to Y$ be an lc-trivial (respectively klt-trivial) fibration. Then $f$ is \emph{acceptable} if there exists another lc-trivial (res\-pec\-ti\-ve\-ly klt-trivial) fibration $\bar f\colon (\overline X,\overline \Delta)\to Y$ such that $\overline \Delta$ is effective on the generic fibre of $\bar f$, and a birational morphism $\mu\colon X\to\overline X$ such that $f=\bar f\circ\mu$ and $K_X+\Delta\sim_\Q\mu^*(K_{\overline X}+\overline \Delta)$. Note that then the horizontal part of $\Delta^{<0}$ with respect to $f$ is $\mu$-exceptional. Note also that any birational base change of an acceptable lc-trivial (respectively klt-trivial) fibration is again an acceptable lc-trivial (respectively klt-trivial) fibration.
$$
\xymatrix{
(X,\Delta)\ar[dr]_f \ar[r]^\mu & (\overline X,\overline \Delta)\ar[d]^{\bar f}\\
& Y
}
$$}
\end{dfn}

\begin{lem}\label{relSNC}
Let  $f\colon(X, \Delta)\rightarrow Y$ be an acceptable lc-trivial (respectively klt-trivial) fibration, where $(X,\Delta)$ is log smooth and $Y$ is a smooth Ambro model for $f$. Fix a prime divisor $T$ on $Y$. Then there exists a birational morphism $\alpha\colon Y'\to Y$ from a smooth variety $Y'$ and a commutative diagram
$$
\xymatrix{
X\ar[d]_{f} & X_\nu \ar[l]^{\alpha_\nu} \ar[d]_{f_\nu} & X' \ar[ld]^{f'}\ar[l]^{\ \beta}\\
Y & Y', \ar[l]_\alpha&
}
$$
such that $X'$ is a smooth variety, $X_\nu$ is the normalisation of the main component of $X\times_{Y} Y'$, $\beta$ is birational and, if $\Delta'$ is defined by $K_{X'}+\Delta'=\pi^*(K_X+\Delta)$ and $T'=\alpha_*^{-1}T$, then:
\begin{enumerate}[label=(\alph*)]
\item \label{a} $f'\colon (X',\Delta')\to Y'$ is an acceptable lc-trivial (respectively klt-trivial) fibration,
\item \label{b} there exists an $(f',T')$-bad divisor $\Sigma_{f',T'}\subseteq Y'$ which has simple normal crossings, and
\item \label{c} the divisor $\Delta'+f'^*\Sigma_{f',T'}$ has simple normal crossings support.
\end{enumerate}

Moreover, we may define $\alpha$ as an embedded resolution of any $(f,T)$-bad divisor $\Sigma_{f,T}$ in $Y$, and we may assume that $\beta$ is an isomorphism away from $\Supp(f_\nu^*\alpha^*\Sigma_{f,T})$.
\end{lem}

\begin{proof}
Let $\Sigma_{f,T}$ be an $(f,T)$-bad divisor on $Y$. Let $\alpha\colon Y'\to Y$ be an embedded resolution of the pair $(Y,\Sigma_{f,T})$ as in \S\ref{subsec:pairs}; in particular, $\alpha$ is an isomorphism away from $\Sigma_{f,T}$. Define a divisor $\Delta_\nu$ on $X_\nu$ by
$$K_{X_\nu}+\Delta_\nu=\alpha_\nu^*(K_X+\Delta).$$
Define $\beta\colon X'\to X_\nu$ as an embedded resolution of the pair $\big(X_\nu,\Delta_\nu+f_\nu^*\alpha^*\Sigma_{f,T}\big)$, so in particular, it is an isomorphism away from $\Supp(f_\nu^*\alpha^*\Sigma_{f,T})$.

If $\bar f\colon (\overline{X},\overline{\Delta})\to Y$ is a fibration as in Definition \ref{dfn:acceptable}, then $f'$ factors through the normalisation $\overline{X}'$ of the main component of the fibre product $\overline{X}\times_Y Y'$. Let $\bar f'\colon \overline X'\to Y$ and $\overline \alpha'\colon \overline X'\to \overline X$ be the induced morphisms. If we set $\overline \Delta':=K_{\overline X'}-\overline\alpha'^*(K_{\overline X}+\overline{\Delta})$, then $\overline \Delta'$ is effective on the generic fibre of $\bar f'$. Thus, $f'$ is acceptable, which shows \ref{a}.

We claim that $\Sigma_{f',T'}:=\alpha^{-1}(\Sigma_{f,T})$ is the desired $(f',T')$-bad divisor; the claim clearly implies \ref{b} and \ref{c} by construction.

To show the claim, clearly the set $\alpha^{-1}(\Sigma_{f,T})$ contains the non-smooth locus of $f'$ and the divisor $T'$. Since $Y$ is an Ambro model, we have $K_{Y'}+B_{f'}=\alpha^*(K_Y+B_f)$, hence
$$\Supp B_{f'}\subseteq\alpha^{-1}(\Supp B_f)\cup\Exc\alpha\subseteq\alpha^{-1}(\Sigma_{f,T}).$$

Now, denote $\alpha':=\alpha_\nu\circ\beta\colon X'\to X$. Let $D'$ be a component of the vertical part of $\Delta'$ with respect to $f'$. It remains to show that
\begin{equation}\label{eq:1000}
f'(D')\subseteq\Sigma_{f',T'}.
\end{equation}

To this end, if $D'$ is not $\alpha'$-exceptional, then $D:=\alpha'(D')$ is a component of $\Delta_v$. Thus, $f(D)=\alpha(f'(D'))$ belongs to $\Sigma_{f,T}$ by definition, which implies \eqref{eq:1000}. If $D'$ is $\alpha'$-exceptional, then, since $\alpha'$ is an isomorphism away from $f^{-1}(\Sigma_{f,T})$, we have $\alpha'(D')\subseteq f^{-1}(\Sigma_{f,T})$. Therefore, $f(\alpha'(D'))=\alpha(f'(D'))$ is a subset of $\Sigma_{f,T}$, which gives \eqref{eq:1000} and finishes the proof.
\end{proof}

\section{Semiampleness on divisorial valuations} \label{sec:5and6}

In this section we prove the main technical results of this paper, Propositions \ref{pro:step1aa} and \ref{pro:step1ab}. At the end of the section, we then deduce Theorems \ref{thm:main} and \ref{thm:bplus}, as well as Corollaries \ref{cor:surfaces} and \ref{cor:threefolds}.

\begin{lem}\label{lem:mincentreacceptable}
Let  $f\colon(X, \Delta)\rightarrow Y$ be an lc-trivial fibration, where $(X,\Delta)$ is a log smooth log canonical pair and $Y$ is a smooth Ambro model for $f$. Fix a prime divisor $T$ on $Y$. Assume that there exists an $(f,T)$-bad divisor $\Sigma_{f,T}\subseteq Y$ which has simple normal crossings, and such that the divisor $\Delta+f^*\Sigma_{f,T}$ has simple normal crossings support. Denote
$$\Delta_X=\Delta+\sum_{\Gamma\subseteq \Sigma_{f,T}}\gamma_{\Gamma}f^*\Gamma,$$
where $\gamma_\Gamma$ are the generic log canonical thresholds with respect to $f$ as in Definition \ref{dfn:cbf}. Then $(X,\Delta_X)$ is a log smooth log canonical pair and there exists a smooth minimal log canonical centre of $(X,\Delta_X)$ over $T$.
\end{lem}

\begin{proof}
We first note that
$$\textstyle K_X+\Delta_X\sim_\Q f^*(K_Y+\Sigma_{f,T}+M_f),$$
and the pair $(Y,\Sigma_{f,T})$ is log canonical since $\Sigma_{f,T}$ is a simple normal crossings divisor. Therefore, by Theorem \ref{thm:invAdjunction}, the pair $(X,\Delta_X)$ is log canonical, and it is not klt over each prime divisor $\Gamma\subseteq\Sigma_{f,T}$. Moreover, the support of $\Delta_X$ has simple normal crossings by assumption.

Thus, there exists a component $D$ of $f^*T$ which dominates $T$ and which has coefficient $1$ in $\Delta_X$. In other words, there exists a log canonical centre of $(X,\Delta_X)$ which dominates $T$. In particular, there exists a minimal log canonical centre of $(X,\Delta_X)$ over $T$. Since all log canonical centres of $(X,\Delta_X)$ are connected components of intersections of components of $\Delta_X^{=1}$, they are all smooth.
\end{proof}

\begin{pro}\label{pro:step1aa}
Let  $f\colon(X, \Delta)\rightarrow Y$ be an acceptable lc-trivial fibration, where $(X,\Delta)$ is a log smooth log canonical pair and $Y$ is a smooth Ambro model for $f$. Fix a prime divisor $T$ on $Y$. Assume that there exists an $(f,T)$-bad divisor $\Sigma_{f,T}\subseteq Y$ which has simple normal crossings, and such that the divisor $\Delta+f^*\Sigma_{f,T}$ has simple normal crossings support. Denote
$$\Delta_X=\Delta+\sum_{\Gamma\subseteq \Sigma_{f,T}}\gamma_{\Gamma}f^*\Gamma,$$
where $\gamma_\Gamma$ are the generic log canonical thresholds with respect to $f$ as in Definition \ref{dfn:cbf}. Denote
$$\Xi_T:=(\Sigma_{f,T}-T)|_T.$$
Let $S$ be a minimal log canonical centre of $(X, \Delta_X)$ over $T$, which exists by Lemma \ref{lem:mincentreacceptable}. Let
$$f|_S\colon S\overset{h}{\longrightarrow} T'\overset{\tau}{\longrightarrow} T$$
be the Stein factorisation, and let $R$ denote the ramification divisor of $\tau$ on $T'$. Then:
\begin{enumerate}
\item[(i)] if $K_S+\Delta_S=(K_X+ \Delta_X)|_S$, then $h\colon (S, \Delta_S)\rightarrow T'$ is a klt-trivial fibration with $B_h\geq0$,
\item[(ii)] $\tau^*(M_f|_T)\sim_\Q M_h+R'+E$, where $M_f$ is chosen so that $T\not\subseteq M_f$ and
$$R'=\sum\limits_{\Gamma\not\subseteq \tau^{-1}(\Xi_T)}(\mult_\Gamma R)\cdot \Gamma\quad\text{and}\quad E=\sum\limits_{\Gamma\not\subseteq\tau^{-1}(\Xi_T)}(\mult_\Gamma B_h)\cdot\Gamma.$$
\end{enumerate}
\end{pro}

\begin{proof}
\emph{Step 1.}
As in the proof of Lemma \ref{lem:mincentreacceptable}, we first note that
\begin{equation}\label{eq:6}
\textstyle K_X+\Delta_X\sim_\Q f^*\big(K_Y+\Sigma_{f,T}+M_f\big),
\end{equation}
and that $(X,\Delta_X)$ is a log smooth log canonical pair. By restricting the equation \eqref{eq:6} to $S$ we obtain
$$K_S+\Delta_S\sim_\Q(f|_S)^*(K_T+\Xi_T+M_f|_T),$$
hence $h\colon(S,\Delta_S)\to T'$ is an lc-trivial fibration, and moreover, it is a klt-trivial fibration. Indeed, if there existed a log canonical centre $\Theta$ of $(S,\Delta_S)$ which dominated $T'$, then $\Theta$ would be a log canonical centre of $(X,\Delta_X)$ by Proposition \ref{pro:fujino}, which contradicts the minimality of $S$. This shows the first part of  (i).

\medskip

\emph{Step 2.}
Let $\mathcal{G}$ be the set of all components $P$ of $f^*\Sigma_{f,T}$ with $\mult_P\Delta_X=0$ and denote
$$G:=\Supp\Delta_{X,h}^{<0}\cup\Supp\Delta_{X,v}^{<1}\cup\bigcup_{P\in\mathcal G}P\,;$$
we consider $G$ as a reduced divisor on $X$. Let $\mu\colon X\to\overline X$ and $\bar f\colon\overline{X}\to Y$ be the maps given by Definition \ref{dfn:acceptable}. We denote by $\Delta_{X,h}$ and $\Delta_{X,v}$ the horizontal and vertical parts of $\Delta_X$ with respect to $f$.

We claim that the divisor $G$ is a $(\mu,\bar f)$-EWE divisor.

Indeed, by our hypothesis, the divisor $\Supp\Delta_{X,h}^{<0}=\Supp\Delta_{h}^{<0}$ is $\mu$-exceptional. Now, pick a prime divisor $P\subseteq \Supp\Delta_{X,v}^{<1}$. If $f(P)$ has codimension 2 in $Y$, then $P$ is $f$-exceptional. If $Q:=f(P)$ has codimension 1, then $Q\subseteq\Sigma_{f,T}$. By the definition of $\Delta_X$, there exists a prime divisor $E\subseteq f^{-1}(Q)$ dominating $Q$ such that $\mult_E\Delta_X=1$. In particular, $E\not\subseteq\Supp\Delta_{X,v}^{<1}$ and $f(E)=Q$. This shows that $\Supp \Delta_{X,v}^{<1}$ is weakly exceptional over $Y$ and finishes the proof of the claim.

\medskip

\emph{Step 3.}
Pick $\varepsilon\in\Q$ with $0<\varepsilon \ll 1$ such that the pair $(X, \Delta_X^{\geq 0}+\varepsilon G)$ is dlt and denote $F=-\Delta_X^{\leq 0}+\varepsilon G$. Note that
$$\textstyle K_X+\Delta_X^{\geq 0}+\varepsilon G\sim_{\Q,Y} F\geq 0.$$
By Step 2 and by Proposition \ref{pro:EWE} we have
\begin{equation}\label{continuity2}
N_\sigma(F;X/Y)=F.
\end{equation}
We run the $(K_X+\Delta_X^{\geq 0}+\varepsilon G)$-MMP with scaling of an ample divisor over $Y$. By \eqref{continuity2} and by Lemma \ref{lem:druel} this MMP terminates and contracts all the components of $F$. Let $\rho\colon X\dashrightarrow W$ be the resulting birational contraction and let $\psi\colon W\to Y$ be the resulting morphism. Denote $\Delta_W:=\rho_*\Delta_X^{\geq0}\geq0$.
$$
\xymatrix{
(X,\Delta_X^{\geq 0}+\varepsilon G) \ar@{-->}[rr]^{\quad\ \ \rho}\ar[rd]_{f}&& (W,\Delta_W) \ar[ld]^{\psi}\\
 &Y&
}
$$
Therefore, $(W, \Delta_W)$ is $\Q$-factorial dlt pair and note that $\Delta_W=\rho_*\Delta_X$. Moreover, by the definition of $G$,
\begin{equation}\label{eq:Delta}
\text{the divisor }\Delta_{W,v}\text{ is reduced,}
\end{equation}
and
\begin{equation}\label{eq:smallereq}
(\psi^*\Sigma_{f,T})_{\red} \leq \Delta_{W,v}.
\end{equation}
By \eqref{eq:6} we have
\begin{equation}\label{eq:6a}
K_W+\Delta_W\sim_\Q\psi^*(K_Y+\Sigma_{f,T}+M_f),
\end{equation}
hence $\psi\colon (W, \Delta_W)\rightarrow Y$ is an lc-trivial fibration and the map $\rho\colon (X, \Delta_X)\dasharrow (W, \Delta_W)$ is crepant birational. By Remark \ref{bbir}, we have
\begin{equation}\label{eq:equal}
B_\psi=\Sigma_{f,T}\quad\text{and}\quad M_\psi=M_f.
\end{equation}

\medskip

\emph{Step 4.}
By Lemma \ref{lem:1to1} the map $\rho$ is an isomorphism at the generic point of each log canonical centre of $(X,\Delta_X)$. In particular, there is a log canonical centre $S_W$ of $(W,\Delta_W)$ which is the strict transform of $S$, and it is a minimal log canonical centre of $(W,\Delta_W)$ over $T$. Let
$$\psi|_{S_W}\colon S_W\overset{h_W}{\longrightarrow} T_W\overset{\tau_W}{\longrightarrow} T$$
be the Stein factorisation.

We claim that $T_W=T'$ and $\tau_W=\tau$ (up to isomorphism). To this end, let $(p,q)\colon Z\to S\times S_W$ be the resolution of indeterminacies of the birational map $\rho|_S\colon S\dashrightarrow S_W$. Since $S$ and $S_W$ are normal by Proposition \ref{pro:fujino}, both $p$ and $q$ have connected fibres by Zariski's main theorem. As every curve contracted by $p$ is contracted by $h_W\circ q$, by the Rigidity lemma \cite[Lemma 1.15]{Deb01} there exists a morphism $\xi\colon S\to T_W$ with connected fibres such that $h_W\circ q=\xi\circ p$, and thus $f|_S=\tau_W\circ\xi$. The claim follows by the uniqueness of the Stein factorisation.

\medskip

\emph{Step 5.}
By restricting the equation \eqref{eq:6a} to $S_W$ we obtain
\begin{equation}\label{eq:8}
K_{S_W}+\Delta_{S_W}\sim_\Q(\psi |_{S_W})^*(K_T+\Xi_T+M_f|_T),
\end{equation}
hence $h_W\colon(S_W,\Delta_{S_W})\to T'$ is a klt-trivial fibration by a similar argument as in Step 1. Therefore, the map $\rho|_S\colon (S,\Delta_S)\dashrightarrow (S_W,\Delta_{S_W})$ is crepant birational over $T'$, hence by Remark \ref{bbir} we have
\begin{equation}\label{eq:BandM}
B_{h_W}=B_h\quad\text{and}\quad M_{h_W}\sim_\Q M_h.
\end{equation}
The divisor $B_{h_W}$ is effective since $\Delta_{S_W}$ is. This finishes the proof of (i).

\medskip

\emph{Step 6.}
By \eqref{eq:smallereq} there exists a component $D$ of $\psi^*T$ which dominates $T$ and which has coefficient $1$ in $\Delta_W$. Denote $\Delta_D:=(\Delta_W-D)|_D$, so that $\psi|_D\colon (D,\Delta_D)\to T$ is an lc-trivial fibration. Let $P$ be a component of $(\psi|_D)^*\Xi_T$. Since $(\psi|_D)^*\Xi_T=(\psi^*\Sigma_{f,T}-\psi^*T)|_D$, and each component of $\psi^*\Sigma_{f,T}$ is a component of $\Delta_W^{=1}$ by \eqref{eq:Delta} and \eqref{eq:smallereq}, this implies that $P$ is a component of $(\Delta_W^{=1}-D)|_D=\Delta_D^{=1}$. In other words,
\begin{equation}\label{eq:1001}
\big((\psi|_D)^*\Xi_T\big)_{\red}\leq\Delta_D^{=1}.
\end{equation}
Now, by Proposition \ref{pro:fujino} there are components $S_1,\dots,S_k$ of $\Delta_D^{=1}$ such that $S_W$ is a component of $S_1\cap\dots\cap S_k$, and note that the $S_i$ dominate $T$. This and \eqref{eq:1001} imply
$$\big((\psi|_D)^*\Xi_T\big)_{\red}\leq\Delta_D^{=1}-S_1-\dots-S_k,$$
hence
$$\big((\psi|_{S_W})^*\Xi_T\big)_{\red}\leq(\Delta_D^{=1}-S_1-\dots-S_k)|_{S_W}\leq \Delta_{S_W}^{=1}.$$

Thus, for every prime divisor $P\subseteq \Supp \tau^*\Xi_T$, the generic log-canonical threshold $\gamma_P$ of $(S_W,\Delta_{S_W})$ with respect to $h_W^*P$ is zero. If we define
$$E:=\sum\limits_{\Gamma\not\subseteq\tau^{-1}(\Xi_T)}(\mult_\Gamma B_{h_W})\cdot\Gamma=\sum\limits_{\Gamma\not\subseteq\tau^{-1}(\Xi_T)}(\mult_\Gamma B_h)\cdot\Gamma,$$
where the second equality follows from \eqref{eq:BandM}, then
\begin{equation}\label{eq:ef}
B_{h_W}=(\tau^*\Xi_T)_\red+E.
\end{equation}

\medskip

\emph{Step 7.}
Now we have all the ingredients to show (ii). From \eqref{eq:8} we have
\begin{equation}\label{eq:9}
K_{T'}+B_{h_W}+M_{h_W}\sim_\Q \tau^*(K_T+\Xi_T+M_f|_T).
\end{equation}
Write the Hurwitz formula for $\tau$ as $K_{T'}=\tau^*K_T+R$, and define
$$R':=R-\tau^*\Xi_T+(\tau^*\Xi_T)_\red.$$
Then \eqref{eq:ef} gives
$$\tau^*(K_T+\Xi_T)=K_{T'}-R+\tau^*\Xi_T+B_{h_W}-E-(\tau^*\Xi_T)_\red=K_{T'}-R'+B_{h_W}-E,$$
which together with \eqref{eq:BandM} and \eqref{eq:9} yields
$$\tau^*(M_f|_T)\sim_\Q M_h+R'+E.$$
This finishes the proof.
\end{proof}

\begin{rem}
{\rm In the proof of Proposition \ref{pro:step1aa}, the MMP technique we use is similar to (and inspired by) the one in the proof of \cite[Claim on p.\ 1730]{FG14}. The main difference between our approach and the one of \cite{FG14} is that we avoid semistable reduction in codimension 1 by choosing carefully the EWE divisor in Step 2 and by proving \eqref{eq:Delta} and \eqref{eq:smallereq}.}
\end{rem}

\begin{pro}\label{pro:step1ab}
Let  $f\colon(X, \Delta)\rightarrow Y$ be an acceptable lc-trivial fibration, where $(X,\Delta)$ is a log smooth log canonical pair and $Y$ is a smooth Ambro model for $f$. Fix a prime divisor $T$ on $Y$. Assume that there exists an $(f,T)$-bad divisor $\Sigma_{f,T}\subseteq Y$ which has simple normal crossings, and such that the divisor $\Delta+f^*\Sigma_{f,T}$ has simple normal crossings support.

Then there exists a commutative diagram
$$
\xymatrix{
X\ar[d]_{f} & X_0 \ar[d]^{f_0}\ar[l]_{\delta_{0,X}}\\
Y & Y_0, \ar[l]_{\delta_0}&
}
$$
where $\delta_0$ and $\delta_{0,X}$ are projective birational morphisms, such that, if $T_0$ and $\Delta_0$ are defined by $T_0:=(\delta_0)_*^{-1}T\subseteq Y_0$ and $K_{X_0}+\Delta_0=\delta_{0,X}^*(K_X+\Delta)$, then the following holds.

There exists an $(f_0,T_0)$-bad divisor $\Sigma_{f_0,T_0}\subseteq Y_0$ which has simple normal crossings, and such that the divisor $\Delta_0+f_0^*\Sigma_{f_0,T_0}$ has simple normal crossings support. Denote
$$\Delta_{X_0}=\Delta_0+\sum_{\Gamma\subseteq \Sigma_{f_0,T_0}}\gamma_{\Gamma}f_0^*\Gamma,$$
where $\gamma_\Gamma$ are the generic log canonical thresholds with respect to $f_0$ as in Definition \ref{dfn:cbf}. Let $S_0$ be a minimal log canonical centre of $(X_0, \Delta_{X_0})$ over $T_0$, which exists by Lemma \ref{lem:mincentreacceptable}. Let $f_0|_{S_0}\colon S_0\overset{h_0}{\longrightarrow} T_0'\overset{\tau_0}{\longrightarrow} T_0$ be the Stein factorisation, and let $h_0\colon (S_0, \Delta_{S_0})\rightarrow T_0'$ be the klt-trivial fibration as in Proposition \ref{pro:step1aa}(i). Then:
\begin{enumerate}
\item[(i)] $T_0'$ is an Ambro model for $h_0$,
\item[(ii)] $\tau_0^*(M_{f_0}|_{T_0})\sim_\Q M_{h_0}$, where $M_{f_0}$ is chosen so that $T_0\not\subseteq M_{f_0}$.
\end{enumerate}
\end{pro}

\begin{proof}
We use the notation from Proposition \ref{pro:step1aa}.

\medskip

\emph{Step 1.}
Let $R'$ and $E$ be the divisors as in Proposition \ref{pro:step1aa}(ii), and consider the closed subset
$$\Pi:=\tau\big(\Supp(R'+E)\big)\subseteq T\subseteq Y.$$
Let $Y_b\to Y$ be the blowup of $Y$ along $\Pi$, and let $Y_0\to Y_b$ be an embedded resolution of the strict transform of $T$ in $Y_b$, see \S\ref{subsec:pairs}. Let
$$\delta_0\colon Y_0\to Y$$
be the composition. Then, in particular,
\begin{equation}\label{eq:delta}
\delta_0\big(\Exc(\delta_0)\big)=\Pi.
\end{equation}
Let $X_\nu$ be the normalisation of the main component of the fibre product $X\times_Y Y_0$ with the induced morphism $\delta_\nu\colon X_\nu\to X$, and define a divisor $\Delta_\nu$ on $X_\nu$ by $K_{X_\nu}+\Delta_\nu=\delta_\nu^*(K_X+\Delta)$. Let $\beta\colon X_0\to X_\nu$ be an embedded resolution of $\big(X_\nu,\Delta_\nu+(f\circ\delta_\nu)^{-1}(\Sigma_{f,T})\big)$; in particular, $\beta$ is an isomorphism away from $(f\circ\delta_\nu)^{-1}(\Pi)$.

Set $\delta_{0,X}:=\delta_\nu\circ\beta\colon X_0\to X$, so that we obtain the commutative diagram
$$
\xymatrix{
X\ar[d]_{f} & X_\nu \ar[l]^{\delta_\nu} \ar[d] & X_0 \ar[ld]^{f_0}\ar[l]^\beta \ar@/_0.8pc/[ll]_{\delta_{0,X}}\\
Y & Y_0. \ar[l]_{\delta_0}&
}
$$
Define a divisor $\Delta_0$ on $X_0$ by $K_{X_0}+\Delta_0=\beta^*(K_{X_\nu}+\Delta_\nu)$. Then $(X_0,\Delta_0)$ is a log smooth pair, and let $f_0\colon (X_0,\Delta_0)\to Y_0$ be the induced lc-trivial fibration.

Let $T_0:=(\delta_0)_*^{-1}T\subseteq Y_0$. Then the proof of Lemma \ref{relSNC} shows that the divisor
$$\Sigma_{f_0,T_0}:=\delta_0^{-1}(\Sigma_{f,T})$$
is an $(f_0,T_0)$-bad divisor in $Y_0$ which has simple normal crossings and such that the support of the divisor $\Delta_{X_0}:=\Delta_0+f_0^*\Sigma_{f_0,T_0}$ has simple normal crossings.

Since $Y$ is an Ambro model for $f$, we have $M_{f_0}=\delta_0^*M_f$, hence
\begin{equation}\label{eq:pullback1}
M_{f_0}|_{T_0}=(\delta_0|_{T_0})^*(M_f|_T).
\end{equation}

\medskip

\emph{Step 2.}
Since $\delta_{0,X}^{-1}$ is an isomorphism at the generic point of $S$, there exists a unique log canonical centre $S_0$ of $(X_0,\Delta_{X_0})$ which is minimal over $T_0$, and such that the map $\delta_{0,X}|_{S_0}\colon S_0\to S$ is birational. Let
$$f_0|_{S_0}\colon S_0 \overset{h_0}{\longrightarrow} T_0'\overset{\tau_0}{\longrightarrow} T_0$$
be the Stein factorisation. Then $h_0\colon (S_0,\Delta_{S_0})\to T_0'$ is a klt-trivial fibration as in Proposition \ref{pro:step1aa}(i), where $K_{S_0}+\Delta_{S_0}=(K_{X_0}+\Delta_{X_0})|_{S_0}$.

The Rigidity lemma \cite[Lemma 1.15]{Deb01} applied to the diagram
$$
\xymatrix{
S\ar[d]_h && S_0 \ar[ll]_{\delta_{0,X}|_{S_0}} \ar[d]^{h_0} \\
T'\ar[d]_\tau && T_0' \ar[d]^{\tau_0} \\
T && T_0 \ar[ll]_{\delta_0|_{T_0}}
}
$$
shows that there is a morphism $\delta_0'\colon T_0'\to T'$ making the diagram commutative; note that the morphism is then necessarily birational. Thus, we obtain the commutative diagram
\begin{equation}\label{dia1}
\begin{gathered}
\xymatrix{
T'\ar[d]_{\tau} & T_0' \ar[l]_{\ \delta_0'} \ar[d]^{\tau_0} \\
T & T_0 \ar[l]_{\delta_0|_{T_0}}
}
\end{gathered}
\end{equation}
and
\begin{equation}\label{statement1}
\text{$M_{h_0}$ is the moduli divisor obtained by the base change of $h$ by $\delta_0'$.}
\end{equation}

\medskip

\emph{Step 3.}
We claim that
\begin{equation}\label{eq:moduli1}
\tau_0^*(M_{f_0}|_{T_0})\sim_\Q M_{h_0},
\end{equation}
which then shows (ii).

To this end, denote $\Xi_{T_0}:=(\Sigma_{f_0,T_0}-T_0)|_{T_0}$, and let $R_0$ be the ramification divisor of $\tau_0$. Let $\Gamma$ be a prime divisor in $\Supp R_0\cup\Supp B_{h_0}$. Then by Proposition \ref{pro:step1aa}(ii) it suffices to show that
\begin{equation}\label{eq:help}
\Gamma\subseteq \tau_0^{-1}(\Xi_{T_0}).
\end{equation}

There are two cases. Assume first that $\tau_0(\Gamma)$ -- viewed as a closed subset of $Y_0$ -- is a subset of $\Exc(\delta_0)$. Then, since $\Exc(\delta_0)$ is a divisor, there exists a prime divisor $\overline{\Gamma}\subseteq\Exc(\delta_0)$ such that $\tau_0(\Gamma)\subseteq \overline{\Gamma}\cap T_0$. Since $\Exc(\delta_0)\subseteq\Supp \Sigma_{f_0,T_0}$ by construction, we have
$$\tau_0(\Gamma)\subseteq \Supp (\Sigma_{f_0,T_0}-T_0)\cap T_0=\Xi_{T_0},$$
which implies \eqref{eq:help}.

Assume now that $\tau_0(\Gamma)\not\subseteq\Exc(\delta_0)$. Then by \eqref{eq:delta} we have $\delta_0(\tau_0(\Gamma))\not\subseteq\Pi$, and hence by \eqref{dia1},
\begin{equation}\label{eq:1111}
\delta_0'(\Gamma)\not\subseteq \Supp(R'+E).
\end{equation}
Since $\delta_0$ is an isomorphism at the generic point of $\tau_0(\Gamma)$, in the neighbourhood of this generic point we have
$$\delta_0'(\Supp R_0)=\Supp R\quad\text{and}\quad \delta_0'(\Supp B_{h_0})=\Supp B_h.$$
Therefore,
\begin{equation}\label{eq:1112}
\delta_0'(\Gamma)\subseteq \Supp R\cup\Supp B_h.
\end{equation}
But now \eqref{eq:1111} and \eqref{eq:1112} imply that $\delta_0'(\Gamma)\subseteq\tau^{-1}(\Xi_T)$, hence there exists a prime divisor $\overline{\Gamma}\subseteq\Supp(\Sigma_{f,T}-T)$ such that $\tau(\delta_0'(\Gamma))\subseteq \overline{\Gamma}\cap T$. Since $\tau(\delta_0'(\Gamma))=\delta_0(\tau_0(\Gamma))$ by \eqref{dia1}, and $\delta_0$ is an isomorphism at the generic point of $\tau_0(\Gamma)$, this implies
$$\tau_0(\Gamma)\subseteq(\delta_0)_*^{-1}\overline{\Gamma}\cap T_0\subseteq \Supp (\Sigma_{f_0,T_0}-T_0)\cap T_0=\Xi_{T_0},$$
which implies \eqref{eq:help} and finishes the proof of \eqref{eq:moduli1}.

\medskip

\emph{Step 4.}
Let $\nu\colon \widetilde T_0\to T_0'$ be an Ambro model for $h_0$. Then by Lemma \ref{lem:blowup} there exists a birational morphism $\mu\colon T_1\to T_0$ obtained by a sequence of blowups such that, if we denote by $\lambda\colon \widehat{T}_0\dashrightarrow \widetilde{T}_0$  the induced birational map from the normalisation of the main component of $T_0'\times_{T_0} T_1$ to $\widetilde{T}_0$, then $\lambda^{-1}$ is an isomorphism at the generic point of each $\nu$-exceptional prime divisor on $\widetilde{T}_0$. We denote by $\delta_1\colon Y_1\to Y_0$ the corresponding birational map of ambient spaces induced by $\mu$, so that $\mu=\delta_1|_{T_1}$.
\begin{equation}\label{dia3}
\begin{gathered}
\xymatrix{
T_0'\ar[d]_{\tau_0} & \widetilde{T}_0 \ar[l]_\nu & \widehat T_0 \ar[d] \ar@{-->}[l]_{\ \lambda} \\
T_0 && T_1 \ar[ll]_{\delta_1|_{T_1}} &
}
\end{gathered}
\end{equation}
By possibly blowing up further, we may assume that $\delta_1$ is an embedded resolution of the pair $(Y_0,\Sigma_{f_0,T_0})$, and in particular, that $\delta_1$ is an isomorphism away from $\Sigma_{f_0,T_0}$.

Let $(X_1,\Delta_1)$ be a log smooth pair which is a crepant pullback of $(X_0,\Delta_0)$ obtained by making a base change of $f_0$ by $\delta_1$ and taking an embedded resolution of the preimage of $\Sigma_{f_0,T_0}$ in $X_0\times_{Y_0} Y_1$ (similarly as in Step 1), and let $f_1\colon (X_1,\Delta_1)\to Y_1$ be the induced klt-trivial fibration.
\begin{equation}\label{dia5}
\begin{gathered}
\xymatrix{
X_0\ar[d]_{f_0} & X_0\times_{Y_0} Y_1 \ar[l] \ar[d] & X_1 \ar[ld]^{f_1}\ar[l] \ar@/_0.8pc/[ll]_{\delta_{1,X}}\\
Y_0 & Y_1 \ar[l]_{\delta_1}&
}
\end{gathered}
\end{equation}
Observe that $T_1=(\delta_1)_*^{-1}T_0$. Then the proof of Lemma \ref{relSNC} shows that the divisor
$$\Sigma_{f_1,T_1}:=\delta_1^{-1}(\Sigma_{f_0,T_0})$$
is an $(f_1,T_1)$-bad divisor in $Y_1$ which has simple normal crossings and such that the support of the divisor $\Delta_{X_1}:=\Delta_1+f_1^*\Sigma_{f_1,T_1}$ has simple normal crossings.

Since $Y_0$ is also an Ambro model, we have $M_{f_1}=\delta_1^*M_{f_0}$, hence
\begin{equation}\label{eq:pullback2}
M_{f_1}|_{T_1}=(\delta_1|_{T_1})^*(M_{f_0}|_{T_0}).
\end{equation}

\medskip

\emph{Step 5.}
Since $\delta_{1,X}^{-1}$ is an isomorphism at the generic point of $S_0$, there exists a unique log canonical centre $S_1$ of $(X_1,\Delta_{X_1})$ which is minimal over $T_1$, and such that the map $\delta_{1,X}|_{S_1}\colon S_1\to S_0$ is birational. Let
$$f_1|_{S_1}\colon S_1 \overset{h_1}{\longrightarrow} T_1'\overset{\tau_1}{\longrightarrow} T_1$$
be the Stein factorisation. Then $h_1\colon (S_1,\Delta_{S_1})\to T_1'$ is a klt-trivial fibration as in Proposition \ref{pro:step1aa}(i), where $K_{S_1}+\Delta_{S_1}=(K_{X_1}+\Delta_{X_1})|_{S_1}$.

As in Step 2, there is a birational morphism $\delta_1'\colon T_1'\to T_0'$ such that the diagram
\begin{equation}\label{dia2}
\begin{gathered}
\xymatrix{
T_0'\ar[d]_{\tau_0} & T_1' \ar[l]_{\ \delta_1'} \ar[d]^{\tau_1} \\
T_0 & T_1 \ar[l]_{\delta_1|_{T_1}}
}
\end{gathered}
\end{equation}
commutes, and
\begin{equation}\label{statement2}
\text{$M_{h_1}$ is the moduli divisor obtained by the base change of $h_0$ by $\delta_1'$.}
\end{equation}

\medskip

\emph{Step 6.}
We claim that
\begin{equation}\label{eq:moduli2}
\tau_1^*(M_{f_1}|_{T_1})\sim_\Q M_{h_1}.
\end{equation}
To this end, denote $\Xi_{T_1}:=(\Sigma_{f_1,T_1}-T_1)|_{T_1}$, and let $R_1$ be the ramification divisor of $\tau_1$. Let $\Gamma$ be a prime divisor in $\Supp R_1\cup\Supp B_{h_1}$. Then by Proposition \ref{pro:step1aa}(ii) it suffices to show that
\begin{equation}\label{eq:help2}
\Gamma\subseteq \tau_1^{-1}(\Xi_{T_1}).
\end{equation}

There are two cases. If $\tau_1(\Gamma)\subseteq\Exc(\delta_1)$, then we conclude analogously as in Step 3.

Now we assume that $\tau_1(\Gamma)\not\subseteq\Exc(\delta_1)$. Then since $\delta_1$ is an isomorphism at the generic point of $\tau_1(\Gamma)$, in the neighbourhood of this generic point we have
$$\delta_1'(\Supp R_1)=\Supp R_0\quad\text{and}\quad \delta_1'(\Supp B_{h_1})=\Supp B_{h_0}.$$
Therefore,
\begin{equation}\label{eq:1113}
\delta_1'(\Gamma)\subseteq \Supp R_0\cup\Supp B_{h_0}.
\end{equation}
But now \eqref{eq:1113} and \eqref{eq:help} imply that $\delta_1'(\Gamma)\subseteq\tau_0^{-1}(\Xi_{T_0})$, hence there exists a prime divisor $\overline{\Gamma}\subseteq\Supp(\Sigma_{f_0,T_0}-T_0)$ such that $\tau_0(\delta_1'(\Gamma))\subseteq \overline{\Gamma}\cap T_0$. Since $\tau_0(\delta_1'(\Gamma))=\delta_1(\tau_1(\Gamma))$ by \eqref{dia2}, and $\delta_1$ is an isomorphism at the generic point of $\tau_1(\Gamma)$, this implies
$$\tau_1(\Gamma)\subseteq(\delta_1)_*^{-1}\overline{\Gamma}\cap T_1\subseteq \Supp (\Sigma_{f_1,T_1}-T_1)\cap T_1=\Xi_{T_1},$$
which shows \eqref{eq:help2} and finishes the proof of \eqref{eq:moduli2}.

\medskip

\emph{Step 7.}
Recall that $\widehat{T}_0$ is the main component of $T_0'\times_{T_0} T_1$. Then from \eqref{dia2} there exists a birational morphism $\xi\colon T_1'\to \widehat{T}_0$, and denote
$$\theta=\lambda\circ\xi\colon T_1'\dashrightarrow \widetilde{T}_0.$$
Then $\theta^{-1}$ is an isomorphism at the generic point of each $\nu$-exceptional prime divisor on $\widetilde{T}_0$. Let us consider $(p,q)\colon V\to \widetilde T_0\times T_1'$ a resolution of indeterminacies of $\theta$. Then by \eqref{dia3} and \eqref{dia2} we have the diagram
\begin{equation}\label{dia4}
\begin{gathered}
\xymatrix{
T_0'\ar[d]_{\tau_0} & \widetilde{T}_0 \ar[l]_{\ \nu} &  T_1' \ar@{-->}[l]_{\quad \theta} \ar[d]^{\tau_1} \ar@/^0.8pc/[ll]^{\delta_1'} & V \ar[l]^q  \ar@/_1pc/[ll]_{p} \\
T_0 && T_1. \ar[ll]^{\delta_1|_{T_1}} &
}
\end{gathered}
\end{equation}

Denote by $M_{\widetilde{T}_0}$ be the moduli divisor of the klt-trivial fibration obtained by the base change of $h_0$ by $\nu$, and denote by $M_V$ be the moduli divisor of the klt-trivial fibration obtained by the base change of $h_1$ by $q$. By \eqref{statement2} and from the diagram \eqref{dia4} we have that $M_V$ is the moduli divisor obtained from $M_{\widetilde{T}_0}$ by the base change by $p$, hence
\begin{equation}\label{eq:ambroforV}
M_V=p^*M_{\widetilde{T}_0}
\end{equation}
since $\widetilde{T}_0$ is an Ambro model for $h_0$.

In particular, $M_{\widetilde{T}_0}$ and $M_V$ are nef $\Q$-divisors. Then, since $\nu_*M_{\widetilde{T}_0}=M_{h_0}$ and $q_*M_V=M_{h_1}$, and since $M_{h_0}$ and $M_{h_1}$ are $\Q$-Cartier divisors by \eqref{eq:moduli1} and \eqref{eq:moduli2}, by the Negativity lemma \cite[Lemma 3.39]{KM98} there exist a $\nu$-exceptional divisor $E_\nu\geq0$ on $\widetilde T_0$ and a $q$-exceptional divisor $E_q\geq0$ on $V$ such that
\begin{equation}\label{eq:twoequations}
M_{\widetilde T_0}=\nu^*M_{h_0}-E_\nu\quad\text{and}\quad M_{V}=q^*M_{h_1}-E_q.
\end{equation}
Therefore:
\begin{align*}
M_{V}&=q^*M_{h_1}-E_q & \text{by \eqref{eq:twoequations}}\\
&\sim_\Q q^*\tau_1^*(M_{f_1}|_{T_1})-E_q & \text{by \eqref{eq:moduli2}}\\
&=q^*\tau_1^*(\delta_1|_{T_1})^*(M_{f_0}|_{T_0})-E_q & \text{by \eqref{eq:pullback2}}\\
&=p^*\nu^*\tau_0^*(M_{f_0}|_{T_0})-E_q & \text{by \eqref{dia4}}\\
&\sim_\Q p^*\nu^*M_{h_0}-E_q & \text{by \eqref{eq:moduli1}}\\
&=p^*M_{\widetilde{T}_0}+p^*E_\nu-E_q & \text{by \eqref{eq:twoequations}}\\
&=M_V+p^*E_\nu-E_q. & \text{by \eqref{eq:ambroforV}}
\end{align*}
Therefore $p^*E_\nu\sim_\Q E_q$. Since $\theta^{-1}$ is an isomorphism at the generic point of each component of $E_\nu$, and $E_q$ is $q$-exceptional, this implies $E_\nu=E_q=0$. In particular, from \eqref{eq:twoequations} we have
\begin{equation}\label{eq:descent1}
M_{\widetilde T_0}=\nu^*M_{h_0},
\end{equation}
hence $M_{h_0}$ is nef and $T_0'$ is an Ambro model for $h$, which gives (i).
\end{proof}

\medskip

Now we can prove the main results of this paper.

\begin{proof}[Proof of Theorem \ref{thm:main}]
We may assume that $Y=Y'$, so that we have to show that $\nu^*M_Y$ is semiample for every prime divisor $T$ on $Y$ with the normalisation $T^\nu$ and the induced morphism $\nu\colon T^\nu\to Y$. By Remark \ref{rem:redtolc} we may assume that $(X,\Delta)$ is log canonical.

We use the following remark repeatedly in the proof: If $\alpha\colon Z\to Y$ is any birational model and if $T_Z:=\alpha^{-1}_*T$ with the normalisation $T_Z^\nu$ and the induced morphism $\nu_Z\colon T_Z^\nu\to Z$, then it suffices to show that $\nu_Z^*M_Z$ is semiample. Indeed, since $Y$ is an Ambro model, we have $M_Z=\alpha^*M_Y$, hence
$$\nu_Z^*M_Z=(\alpha_{T_Z})^*\nu^*M_Y,$$
where $\alpha_{T_Z}\colon T_Z^\nu\to T^\nu$ is the induced morphism. Thus, $\nu^*M_Y$ is semiample if and only if $\nu^*_Z M_Z$ is semiample by Lemma \ref{geq2}.

By replacing $Y$ by its desingularisation and $T$ by its strict transform, we may assume that $Y$ is smooth, and by replacing $(X,\Delta)$ by its log resolution, we may assume that $f$ is an acceptable lc-trivial fibration such that $(X,\Delta)$ is a log smooth log canonical pair. Again by replacing $Y$, $T$ and $(X,\Delta)$ by higher models, by Lemma \ref{relSNC} we may additionally assume that there exists an $(f,T)$-bad divisor $\Sigma_{f,T}$ which has simple normal crossings and such that the divisor $\Delta+f^*\Sigma_{f,T}$ has simple normal crossings support.

Then Proposition \ref{pro:step1ab} shows that, after replacing $Y$, $T$ and $(X,\Delta)$ by higher models, there exists a simple normal crossings divisor $\Delta_X$ on $X$ such that the pair $(X,\Delta_X)$ is log canonical, and there exists a minimal log canonical centre $S$ of $(X,\Delta_X)$ such that, if $f|_S\colon S\overset{h}{\longrightarrow} T'\overset{\tau}{\longrightarrow} T$ is the Stein factorisation, then:
\begin{enumerate}
\item[(a)] $h\colon (S, \Delta_S)\rightarrow T'$ is a klt-trivial fibration, where
$$K_S+\Delta_S=(K_X+\Delta_X)|_S,$$
\item[(b)] $T'$ is an Ambro model for $h$,
\item[(c)] $\tau^*(M_Y|_T)=M_h$.
\end{enumerate}
It follows by (b) that $M_h$ is semiample since we assume the B-Semiampleness Conjecture in dimension $n-1$, and hence $M_Y|_T$ is semiample by (c) and by Lemma \ref{geq2}. This concludes the proof.
\end{proof}

\begin{proof}[Proof of Theorem \ref{thm:bplus}]
We may assume that $Y=Y'$, so that we have to show that $\nu^*M_Y$ is semiample for every prime divisor $T\subseteq\sB_+(M_Y)$ with the normalisation $T^\nu$ and the induced morphism $\nu\colon T^\nu\to Y$. By Remark \ref{rem:redtolc} we may assume that $(X,\Delta)$ is log canonical.

As in the proof of Theorem \ref{thm:main}, we use the following remark repeatedly in the proof: If $T\subseteq\sB_+(M_Y)$ is a prime divisor, if $\alpha\colon Z\to Y$ is any birational model and if $T_Z:=\alpha^{-1}_*T$ with the normalisation $T_Z^\nu$ and the induced morphism $\nu_Z\colon T_Z^\nu\to Z$, then it suffices to show that $\nu_Z^*M_Z$ is semiample. Note that by \cite[Proposition 2.3]{BBP13} we have
$$\sB_+(\alpha^*M_Y)=\alpha^{-1}\big(\sB_+(M_Y)\big)\cup\Exc(\alpha),$$
so that $T_Z\subseteq \sB_+(\alpha^*M_Y)=\sB_+(M_Z)$.

Again as in the proof of Theorem \ref{thm:main}, by replacing $Y$, $T$ and $(X,\Delta)$ by higher models, we may assume that $f$ is an acceptable lc-trivial fibration such that $(X,\Delta)$ is log smooth, and that there exists an $(f,T)$-bad divisor $\Sigma_{f,T}$ which has simple normal crossings and such that the divisor $\Delta+f^*\Sigma_{f,T}$ has simple normal crossings support.

Then Proposition \ref{pro:step1ab} shows that, after replacing $Y$, $T$ and $(X,\Delta)$ by higher models, there exists a simple normal crossings divisor $\Delta_X$ on $X$ such that the pair $(X,\Delta_X)$ is log canonical, and there exists a minimal log canonical centre $S$ of $(X,\Delta_X)$ such that, if $f|_S\colon S\overset{h}{\longrightarrow} T'\overset{\tau}{\longrightarrow} T$ is the Stein factorisation, then:
\begin{enumerate}
\item[(a)] $h\colon (S, \Delta_S)\rightarrow T'$ is a klt-trivial fibration, where
$$K_S+\Delta_S=(K_X+\Delta_X)|_S,$$
\item[(b)] $T'$ is an Ambro model for $h$,
\item[(c)] $\tau^*(M_Y|_T)=M_h$.
\end{enumerate}
As in Proposition \ref{pro:step1aa}, there exist a klt-trivial fibration $h_W\colon(S_W,\Delta_{S_W})\to T'$ and a crepant birational map $\theta\colon(S,\Delta_S)\dashrightarrow(S_W,\Delta_{S_W})$ over $T'$ such that $\Delta_{S_W}$ is effective and
\begin{equation}\label{eq:MM}
M_{h_W}=M_h.
\end{equation}
Since $T$ is a component of $\sB_+(M_Y)$, the divisor $M_Y|_T$ is not big by Lemma \ref{lem:nakamaye}, hence
\begin{equation}\label{eq:kappa}
\kappa(T',M_{h_W})=\kappa(T',M_h)=\kappa(T,M_Y|_T)\leq n-2.
\end{equation}
By Theorem \ref{ambro1}, there exists a diagram
$$
\xymatrix{
(S_W,\Delta_{S_W})\ar[d]_{h_W} && \big(\widetilde S_W,\Delta_{\widetilde S_W}\big)\ar[d]^{\widetilde h_W}\\
T' & W \ar[l]^{\tau_W}\ar[r]_{\widetilde \tau_W}&\widetilde T'
}
$$
where $\tau_W\colon W\rightarrow T'$ is generically finite, $\widetilde \tau_W \colon W\rightarrow \widetilde T'$ is surjective, $\widetilde T'$ is an Ambro model for the klt-trivial fibration $h_W'\colon\big(\widetilde S_W,\Delta_{\widetilde S_W}\big)\to\widetilde{T}'$, the moduli divisor $M_{\widetilde h_W}$ is big and
\begin{equation}\label{eq:tautau2}
\tau_W^*M_{h_W}=\widetilde\tau_W^*M_{\widetilde h_W}.
\end{equation}
In particular, by \eqref{eq:kappa} and \eqref{eq:tautau2} we have
$$\dim\widetilde{T}'=\kappa\big(\widetilde{T}',M_{\widetilde h_W}\big)=\kappa(T',M_{h_W})\leq n-2.$$
Since we assume the B-Semiampleness Conjecture in dimensions at most $n-2$, the divisor $M_{h_W'}$ is semiample. By \eqref{eq:tautau2} and by Lemma \ref{geq2}, the divisor $M_{h_W}$ is semiample, hence $M_h$ is semiample by \eqref{eq:MM}. By (c) and by Lemma \ref{geq2}, this proves finally that the divisor $M_Y|_T$ is semiample.
\end{proof}

\begin{proof}[Proof of Corollary \ref{cor:surfaces}]
Immediate from Theorem \ref{thm:main} and from \cite[Theorem 0.1]{Amb04}.
\end{proof}

\begin{proof}[Proof of Corollary \ref{cor:threefolds}]
Immediate from Theorem \ref{thm:bplus} and from \cite[Theorem 0.1]{Amb04}.
\end{proof}

\section{Reduction to a weaker conjecture}\label{sec:reduction}

In this section we prove that the B-Semiampleness Conjecture is equivalent to the following much weaker version.

\begin{con}\label{con:weakbsemi}
Let $(X,\Delta)$ be a log canonical pair and let $f\colon (X,\Delta)\to Y$ be a klt-trivial fibration over an $n$-dimensional variety $Y$. If $Y$ is an Ambro model of $f$ and if the moduli divisor $M_Y$ is big, then $M_Y$ is semiample.
\end{con}

The following result implies Theorem \ref{thm:reductiontoweaker}.

\begin{thm}\label{KLTimplLC}
Assume Conjecture \ref{con:weakbsemi} in dimensions at most $n$. Then the B-Semiampleness Conjecture holds in dimension $n$.
\end{thm}

\begin{proof}
\emph{Step 1.}
Let $f\colon(X,\Delta)\rightarrow Y$ be an lc-trivial fibration, where $\dim Y=n$, $Y$ is an Ambro model for $f$, and $\Delta$ is effective over the generic point of $Y$. We may assume that $(X,\Delta)$ is not klt over the generic point of $Y$. By Remark \ref{rem:redtolc} we may assume that $(X,\Delta)$ is log canonical.

Let $\mu\colon X'\to X$ be a log resolution of $(X,\Delta)$ and define $\Delta'$ by the formula $K_{X'}+\Delta'=\mu^*(K_X+\Delta)$. Set $f':=f\circ\mu\colon X'\to Y$, and let $S'$ be a minimal log canonical centre of $(X',\Delta')$ over $Y$.

Let $f'|_{S'}\colon S'\overset{h'}{\longrightarrow} Y_1\overset{\tau_1}{\longrightarrow} Y$ be the Stein factorisation. Let $X_1$ be an embedded resolution of the main component of $X'\times_Y Y_1$ with the induced morphisms $\sigma_1\colon X_1\to X'$ and $f_1\colon X_1\to Y_1 $, and define $\Delta_1$ by the formula $K_{X_1}+\Delta_1=\sigma_1^*(K_{X'}+\Delta')$. Then $f_1\colon (X_1,\Delta_1)\to Y_1$ is an acceptable lc-trivial fibration.
$$
\xymatrix{
X \ar[dr]_f & X'\ar[l]_{\mu} \ar[d]^{f'} & X_1 \ar[d]^{f_1} \ar[l]_{\ \sigma_1}\\
& Y & Y_1\ar[l]_{\ \tau_1}
}
$$

By construction, there is an irreducible component $S_1$ of $\sigma_1^{-1}(S')$ such that $\sigma_1\vert_{S_1}\colon S_1\to S'$ is an isomorphism at the generic point of $S_1$. Then $S_1$ is a minimal log canonical centre of $(X_1,\Delta_1)$ over $Y_1$, and denote $h_1:=f_1|_{S_1}\colon S_1\to Y_1$. Then general fibres of $h'$ and $h_1$ are isomorphic, and in particular, a general fibre of $h_1$ is connected. Since $S_1$ is normal by Proposition \ref{pro:fujino}, by considering the Stein factorisation of $h_1$ we deduce that $h_1$ has connected fibres by Zariski's main theorem.

If we define $\Delta_{S_1}$ by the formula $K_{S_1}+\Delta_{S_1}=(K_{X_1}+\Delta_1)|_{S_1}$, then $h_1\colon (S_1,\Delta_{S_1})\to Y_1$ is a klt-trivial fibration by an argument similar to that in Step 1 of the proof of Proposition \ref{pro:step1aa}. Let $\Sigma_{f_1,0}$ be an $(f_1,0)$-bad divisor such that $\Supp B_{h_1}\subseteq\Sigma_{f_1,0}$.

\medskip

\emph{Step 2.}
By Lemma \ref{relSNC} and its proof, there exist a commutative diagram
$$
\xymatrix{
X_1 \ar[d]_{f_1} & X_2 \ar[d]^{f_2} \ar[l]_{\ \sigma_2}\\
Y_1 & Y_2\ar[l]_{\ \tau_2}
}
$$
where the morphism $\tau_2\colon Y_2\to Y_1$ is an embedded resolution of $(Y_2,\Sigma_{f_1,0})$, and $f_2\colon(X_2,\Delta_2)\to Y_2$ is an acceptable lc-trivial fibration such that $\Sigma_{f_2,0}:=\tau_2^{-1}(\Sigma_{f_1,0})$ is an $(f_2,0)$-bad divisor which has simple normal crossings and such that $\Delta_2+f_2^*\Sigma_{f_2,0}$ has simple normal crossings support. By Lemma \ref{relSNC} we may assume that $\sigma_2^{-1}$ is an isomorphism at the generic point of $S_1$, and let $S_2$ be the strict transform of $S_1$ in $X_2$. Denote $h_2:=f_2|_{S_2}\colon S_2\to Y_2$, and define a divisor $\Delta_{S_2}$ by the formula $K_{S_2}+\Delta_{S_2}=(K_{X_2}+\Delta_2)|_{S_2}$. Then $h_2\colon (S_2,\Delta_{S_2})\to Y_2$ is a klt-trivial fibration similarly as in Step 1.

\medskip

\emph{Step 3.}
Set
$$\Delta_{X_2}:=\Delta_2+\sum_{\Gamma\subseteq \Sigma_{f_2,0}}\gamma_{\Gamma}f^*\Gamma,$$
where $\gamma_\Gamma$ are the generic log canonical thresholds with respect to $f_2$ as in Definition \ref{dfn:cbf}. Let $\mathcal{G}$ be the set of all components $P$ of $f_2^*\Sigma_{f_2,0}$ with $\mult_P\Delta_{X_2}=0$ and denote
$$G:=\Supp\Delta_{X_2,h}^{<0}\cup\Supp\Delta_{X_2,v}^{<1}\cup\bigcup_{P\in\mathcal G}P.$$
Pick $\varepsilon\in\Q$ with $0<\varepsilon \ll 1$ such that the pair $(X_2, \Delta_{X_2}^{\geq 0}+\varepsilon G)$ is dlt and run the $(K_{X_2}+\Delta_{X_2}^{\geq 0}+\varepsilon G)$-MMP $\rho\colon X_2\dashrightarrow W$ with scaling of an ample divisor over $Y_2$. As in Steps 2 and 3 of Proposition \ref{pro:step1aa}, this MMP terminates with a dlt pair $(W,\Delta_W)$ with $\Delta_W\geq0$. Moreover, there exists an lc-trivial fibration $\psi\colon (W,\Delta_W)\to Y_2$ such that the divisor $\Delta_{W,v}$ is reduced and
\begin{equation}\label{eq:smallereq1}
\big(\psi^*\Sigma_{f_2,0}\big)_{\red} \leq \Delta_{W,v}.
\end{equation}
Furthermore, we have
\begin{equation}\label{eq:BandM2}
M_\psi=M_{f_2}\quad\text{and}\quad B_\psi=\Sigma_{f_2,0}.
\end{equation}

By Lemma \ref{lem:1to1}, $\rho$ is an isomorphism at the generic point of $S_2$, and let $S_W$ be the strict transform of $S_2$. Then $S_W$ is a minimal log canonical centre of $(W,\Delta_W)$ over $Y_2$ and define $\Delta_{S_W}$ by the formula $K_{S_W}+\Delta_{S_W}=(K_W+\Delta_W)|_{S_W}$, see Proposition \ref{pro:fujino}. Then $h_W:=\psi|_{S_W}\colon(S_W,\Delta_{S_W})\to Y_2$ is a klt-trivial fibration by a similar argument as in Step 4 of the proof of Proposition \ref{pro:step1aa}. Moreover, as in Step 5 of the proof of Proposition \ref{pro:step1aa}, the map $\rho|_{S_2}\colon (S_2,\Delta_{S_2})\dashrightarrow (S_W,\Delta_{S_W})$ is crepant birational over $Y_2$, and we have
\begin{equation}\label{eq:BandM3}
B_{h_W}=B_{h_2}\quad\text{and}\quad M_{h_W}=M_{h_2}.
\end{equation}

Therefore, by comparing the canonical bundle formulas of $\psi$ and $h_W$ and using \eqref{eq:BandM2} and \eqref{eq:BandM3} we obtain
\begin{equation}\label{eq:compare}
\Sigma_{f_2,0}+M_{f_2}=B_{h_2}+M_{h_2}.
\end{equation}

\medskip

\emph{Step 4.}
By Proposition \ref{pro:fujino} there are components $D_1,\dots,D_k$ of $\Delta_W^{=1}$ such that $S_W$ is a component of $D_1\cap\dots\cap D_k$, and note that the $D_i$ dominate $Y_2$. This and \eqref{eq:smallereq1} imply
$$\big(\psi^*\Sigma_{f_2,0}\big)_{\red}\leq\Delta_W^{=1}-D_1-\dots-D_k,$$
hence
$$\big(h_W^*\Sigma_{f_2,0}\big)_{\red}\leq(\Delta_W^{=1}-D_1-\dots-D_k)|_{S_W}\leq \Delta_{S_W}^{=1}.$$

Thus, for every prime divisor $P\subseteq \Supp \Sigma_{f_2,0}$, the generic log-canonical threshold $\gamma_P$ of $(S_W,\Delta_{S_W})$ with respect to $h_W^*P$ is zero. The divisor $B_{h_W}$ is effective since $\Delta_{S_W}$ is, and therefore, there exists an effective $\Q$-divisor $E$ on $Y_2$ having no common components with $\Sigma_{f_2,0}$ such that $B_{h_W}=\Sigma_{f_2,0}+E$, which together with \eqref{eq:BandM3} yields
\begin{equation}\label{eq:ef1}
B_{h_2}=\Sigma_{f_2,0}+E.
\end{equation}

\medskip

\emph{Step 5.}
We claim that $E=0$. To this end, let $P$ be an irreducible component of $E$. If $P\subseteq\Exc(\tau_2)$, then $P\subseteq \Sigma_{f_2,0}$ by the construction of $\tau_2$ in Step 2, a contradiction.

Now, assume that $P\not\subseteq\Exc(\tau_2)$. Then $\tau_2$ is an isomorphism at the generic point of $\tau_2(P)$, hence in the neighbourhood of this generic point we have $\tau_2(\Supp B_{h_2})=\Supp B_{h_1}$. Therefore, since $P$ is a component of $B_{h_2}$, we obtain $\tau_2(P)\subseteq\Supp B_{h_1}$, hence $\tau_2(P)\subseteq \Sigma_{f_1,0}$ by the choice of $\Sigma_{f_1,0}$ in Step 1. But then $P\subseteq\tau_2^{-1}(\Sigma_{f_1,0})=\Sigma_{f_2,0}$, a contradiction.

\medskip

\emph{Step 6.}
We now have $B_{h_2}=\Sigma_{f_2,0}$ by \eqref{eq:ef1}, and thus
\begin{equation}\label{eq:M}
M_{h_2}=M_{f_2}
\end{equation}
by \eqref{eq:compare}. Since $\Sigma_{f_2,0}$ has simple normal crossings support, the variety $Y_2$ is an Ambro model for $h_2$ by the proof of \cite[Theorem 2.7]{Amb04}. Since the map $\rho|_{S_2}\colon (S_2,\Delta_{S_2})\dashrightarrow (S_W,\Delta_{S_W})$ is crepant birational over $Y_2$, by Remark \ref{bbir} the variety $Y_2$ is an also an Ambro model for $h_W$.

By Theorem \ref{ambro1}, there exists a diagram
$$
\xymatrix{
(S_W,\Delta_{S_W})\ar[d]_{h_W} && (S_W',\Delta_{S_W'})\ar[d]^{h_W'}\\
Y_2 & W \ar[l]^{\tau_W}\ar[r]_{\tau_W'}&Y_2'
}
$$
where $\tau_W\colon W\rightarrow Y_2$ is generically finite, $\tau_W' \colon W\rightarrow Y_2'$ is surjective, $Y_2'$ is an Ambro model for the klt-trivial fibration $h_W'\colon(S_W',\Delta_{S_W'})\to Y_2'$, the moduli divisor $M_{h_W'}$ is big and
\begin{equation}\label{eq:tautau}
\tau_W^*M_{h_W}=\tau_W'^*M_{h_W'}.
\end{equation}
Therefore, $M_{h_W'}$ is semiample by the assumptions of the theorem: indeed, if $M_{h_W}$ is big, we may assume that $\tau_W$ and $\tau_W'$ are isomorphisms and that $(S_W,\Delta_{S_W})=(S_W',\Delta_{S_W'})$. Otherwise, we have $\dim Y_2'<\dim Y_2$. By induction on the dimension, Conjecture \ref{con:weakbsemi} in dimensions at most $\dim Y_2'$ implies the B-Semiampleness Conjecture in dimension $\dim Y_2'$, and this then yields that $M_{h_W'}$ is semiample.

Now we have that $M_{h_W}$ is semiample by \eqref{eq:tautau} and by Lemma \ref{geq2}. Thus, $M_{h_2}$ is semiample by \eqref{eq:BandM3}, and so $M_{f_2}$ is semiample by \eqref{eq:M}. By \cite[Proposition 3.1]{Amb05a} we have $M_{f_2}=(\tau_1\circ\tau_2)^*M_f$, and finally, $M_f$ is semiample again by Lemma \ref{geq2}. This finishes the proof.
\end{proof}

\bibliographymark{References}

\providecommand{\bysame}{\leavevmode\hbox to3em{\hrulefill}\thinspace}
\providecommand{\arXiv}[2][]{\href{https://arxiv.org/abs/#2}{arXiv:#1#2}}
\providecommand{\MR}{\relax\ifhmode\unskip\space\fi MR }
\providecommand{\MRhref}[2]{%
  \href{http://www.ams.org/mathscinet-getitem?mr=#1}{#2}
}
\providecommand{\href}[2]{#2}

\end{document}